\documentclass[fleqn,reqno,12pt,a4paper,final]{amsart}

\usepackage[a4paper,left=30mm,right=30mm,top=30mm,bottom=30mm,marginpar=20mm]{geometry}
\usepackage{amsmath}
\usepackage{amssymb}
\usepackage{amsthm}
\usepackage{amscd}
\usepackage[ansinew]{inputenc}
\usepackage{cite}
\usepackage{bbm}
\usepackage{color}
\usepackage[english=american]{csquotes}
\usepackage[final]{graphicx}
\usepackage{hyperref}
\usepackage{calc}
\usepackage{mathptmx}
\usepackage{mathtools}
\usepackage{showkeys}

\numberwithin{equation}{section}

\newtheoremstyle{thmlemcorr}{10pt}{10pt}{\itshape}{}{\bfseries}{.}{10pt}{{\thmname{#1}\thmnumber{ #2}\thmnote{ (#3)}}}
\newtheoremstyle{thmlemcorr*}{10pt}{10pt}{\itshape}{}{\bfseries}{.}\newline{{\thmname{#1}\thmnumber{ #2}\thmnote{ (#3)}}}
\newtheoremstyle{remexample}{10pt}{10pt}{}{}{\bfseries}{.}{10pt}{{\thmname{#1}\thmnumber{ #2}\thmnote{ (#3)}}}
\newtheoremstyle{ass}{10pt}{10pt}{}{}{\bfseries}{.}{10pt}{{\thmname{#1}\thmnumber{ A#2}\thmnote{ (#3)}}}

\theoremstyle{thmlemcorr}
\newtheorem{theorem}{Theorem}
\numberwithin{theorem}{section}
\newtheorem{lemma}[theorem]{Lemma}

\newtheorem{proposition}[theorem]{Proposition}

\newtheorem{definition}[theorem]{Definition}

\theoremstyle{thmlemcorr*}
\newtheorem{theorem*}{Theorem}
\newtheorem{lemma*}[theorem]{Lemma}
\newtheorem{corollary*}[theorem]{Corollary}
\newtheorem{proposition*}[theorem]{Proposition}
\newtheorem{problem*}[theorem]{Problem}
\newtheorem{conjecture*}[theorem]{Conjecture}
\newtheorem{definition*}[theorem]{Definition}

\theoremstyle{remexample}
\newtheorem{remark}[theorem]{Remark}

\theoremstyle{ass}

\DeclareMathOperator{\cof}{cof}

\newcommand{\T}{\mathbb{T}}

\DeclareMathOperator{\Div}{div}

\newcommand{\abs}[1]{|#1|}

\newcommand{\R}{\mathbb{R}}

\newcommand{\vr}{\varrho}

\newcommand{\vu}{v}
\newcommand{\vh}{b}
\newcommand{\ep}{\varepsilon}

\title{Dissipative measure valued solutions for general  conservation laws}
\author{Piotr Gwiazda, Ond\v rej Kreml
and Agnieszka {\'{S}}wierczewska-Gwiazda}
\thanks{O.K. acknowledges the support of the Neuron Impuls Junior project "Mathematical analysis of hyperbolic conservation laws" in the general framework of RVO: 67985840. This work was partially supported by the Simons - Foundation grant 346300 and the Polish Government MNiSW 2015-2019 matching fund.
A. \'S.-G. acknowledges the support of the National Science Centre, DEC-2012/05/E/ST1/02218. The research was partially supported by the Warsaw
Center of Mathematics and Computer Science. P.G  received support from the National Science Centre (Poland), 2015/18/M/ST1/00075.
}
\date{}

\begin{document}

\begin{abstract}
In the last years measure-valued solutions started to be considered as a relevant notion of solutions if they satisfy the so-called measure-valued -- strong uniqueness principle. This means that they coincide with a strong solution emanating from the same initial data if this strong solution exists.  This property has been examined for many systems of mathematical physics, including incompressible and compressible Euler system, compressible Navier-Stokes system et al. and there are also some results concerning general hyperbolic systems. Our goal is to provide a unified framework for general systems, that would cover the most interesting cases of systems, and most importantly, we give examples of equations, for which the aspect of measure-valued -- strong uniqueness has not been considered before, like incompressible magentohydrodynamics and shallow water magnetohydrodynamics. 
\end{abstract}
\maketitle

\section{Introduction}
The recent work of Brenier, De Lellis and Sz\'ekelyhidi~\cite{BrDeLeSz2011} significantly ennobled measu\-re-valued solutions of systems of fluid dynamics, as well as hyperbolic systems in general. They postulated a new principle surprisingly stating that measure-valued solutions, which were  expected to be non-unique to a large extent,   become unique once we know that a strong solution emanating from the same initial data exists.  In this case both solutions coincide on the time interval of existence of the strong solution. 
%
%
 What they called {\it weak-strong uniqueness for measure-valued solutions} is now usually called {\it measure-valued-strong uniqueness}, or {\it mv-strong uniqueness} for short. We favour the latter term, as it seems more adequate.
 The analysis in the case of incompressible Euler system is complete, as DiPerna and Majda had shown in~\cite{DiMa87} existence of measure-valued solutions to the incompressible Euler system exactly in the class which, per the result of Brenier  et al., possesses the property of mv-strong uniqueness. 
  
Careful analysis of the incompressible Euler system allowed the authors of~\cite{BrDeLeSz2011} to conjecture that an analogue property of mv-strong uniqueness  could hold in a more general setting. They had in fact initiated the studies on mv-strong uniqueness for general hyperbolic systems. Following this path, we  also direct our interest to a hyperbolic system of the form
\begin{equation}\label{system}
\partial_tA(u)+\partial_\alpha F_\alpha(u)=0
\end{equation}
with an initial condition $u(0)=u_0$. Here $u : [0,T]\times{\mathbb T}^d \rightarrow \overline X$, where $X\subset\R^n$ is an open convex set and by $\overline X$ we mean the closure of $X$. Moreover,   $A,F_\alpha: \overline X \rightarrow \R^n$, $\alpha = 1,...,d$, we use the Einstein summation convention and we denote $Q=[0,T]\times{\mathbb T}^d$, where 
${\mathbb T}^d$ is a $d-$dimensional torus. 

In~\cite{BrDeLeSz2011} the authors  
studied system \eqref{system} with $A(u)\equiv u$, however their result holds in a class where no existence result is available (and seems impossible to be proven).
This limitation is not particular only for such general systems, but persists even in special cases, including e.g. compressible Euler system, polyconvex elastodynamics or hyperbolic magnetohydrodynamics. 
The solution is in the form of a classical Young measure only (even satisfying a technical assumption that the first moment of this measure is in $L^\infty(Q)$), not a triple consisting of a classical Young measure and concentration and concentration angle measures.

 In parallel Demoulini et al.~\cite{tzavaras1} proved a corresponding result on mv-strong uniqueness for the system of polyconvex elastodynamics. And again the authors attempted to formulate a more general result for hyperbolic systems. Here the possibility of a concentration measure is allowed in the entropy inequality, not in the weak formulation of the system itself. This approach covers, among others, the case initially considered by the authors, i.e.~the system of polyconvex elastodynamics. 
For this system  the mv-strong uniqueness result is in the class coinciding with the class in which one shows existence of solutions. 
 However, this level of generality is still not sufficient to cover the case of abstract hyperbolic system, as well as e.g. Euler equations, where concentration measure appears also in the weak formulation. 
 
 Therefore there is still a need to dispose of assumptions that solutions satisfy any a priori bounds, and in particular, that a solution consists only of a classical Young measure. We find it of great importance to include possibilities of concentration measures appearing in all terms $A(u)$, the flux $F_\alpha(u)$ and an entropy function. A result on mv-strong uniqueness shall be deemed complete whenever the class of measure-valued solutions agrees with the class of an existence result. 
 
 Finally, we give a couple of examples of systems, for which the general result statement gives an original result of mv-strong uniqueness property, namely a system of shallow water magnetohydrodynamics described in Section~\ref{App} and incompressible magnetohydrodynamics described in Section~\ref{Extension}. Surely the list of new applications is not complete.

\subsection{Hypothesis}

Throughout the paper we will assume the following conditions hold.
\begin{enumerate}
\item[(H1)] There exists an open set $X \subset \R^n$ such that the mapping $A : \overline{X} \rightarrow \R^n$ is a $C^2$  map on $X$, continuous on $\overline X$  and satisfies
\begin{equation}\label{H1}
\nabla A(u) \quad \text{ is nonsingular } \forall u \in X. 
\end{equation}
\item[(H2)] The system \eqref{system} is endowed with a companion law
\begin{equation}\label{entropy}
\partial_t \eta(u) + \partial_\alpha q_\alpha(u) = 0
\end{equation}
with an entropy $\eta: \overline X\to \R_+$ which is 
a $C^2$  map on $X$, continuous on $\overline X$  and satisfies
 $\eta(u) \geq 0$ 
and 
\begin{equation}\label{eta-infty}
\lim\limits_{|u|\to\infty}\eta(u)=\infty.
\end{equation}
This yields the existence of a smooth function $G: X \rightarrow \R^n$ such that 
\begin{align}
\nabla \eta &= G \cdot \nabla A \label{H21}\\
\nabla q_\alpha &= G \cdot \nabla F_\alpha, \qquad \alpha = 1,...,d. \label{H22}
\end{align}
The conditions \eqref{H21}-\eqref{H22} are equivalent to
\begin{align}
\nabla G^T \nabla A &= \nabla A^T \nabla G \label{H23}\\
\nabla G^T \nabla F_\alpha &= \nabla F_\alpha^T \nabla G, \qquad \alpha = 1,...,d. \label{H24}
\end{align}
\item[(H3)] The symmetric matrix 
\begin{equation}\label{H3}
\nabla^2 \eta(u) - G(u)\cdot\nabla^2 A(u)
\end{equation}
is positive definite for all $u \in X$.
\item[(H4)] The vector $A(u)$ and the fluxes $F_\alpha(u)$ are bounded by the entropy, i.e.
\begin{align}\label{H41}
|A(u)|&\le C\eta(u) \\
|F_\alpha(u)|&\le C\eta(u), \qquad \alpha = 1,...,d. \label{H42}
\end{align}
\item[(H5)] Defining for a strong solution $U$ taking values in a compact subset of $X$ the relative entropy
\begin{align}
\eta ( u | U ) &:= \eta (u) - \eta (U) - \nabla \eta (U)  \cdot \nabla A (U)^{-1} (A(u) - A(U)) \\ \nonumber
 &= \eta (u) - \eta (U) - G(U) \cdot (A(u) - A(U))
\end{align}
and defining the relative flux as 
\begin{equation}\label{estimates}
F_\alpha (u | U) := F_\alpha (u) - F_\alpha (U) - \nabla F_\alpha (U)  \nabla A (U)^{-1} (A(u) - A(U))
\end{equation}
for $\alpha = 1,...,d$ we assume it holds 
\begin{equation}\label{bound}
|F_\alpha (u | U)|\le C \eta ( u | U ). 
\end{equation}
\end{enumerate}
\begin{remark}\label{remark11}
Observe that in the above definitions the relative flux $F_\alpha (\cdot | U)$ and relative entropy $\eta ( \cdot | U )$
are continuous functions in $\overline X$. This follows directly from the continuity of $F_\alpha (\cdot)$ and 
$\eta ( \cdot)$. Note that there is an asymmetry, the relative functions are well defined for $u\in\overline X$, but for $U\in X$. 
\end{remark}
\begin{remark}\label{remark}
Note that if instead of (H4) we assume that 
\begin{equation}\label{H4thanos}
\lim_{\abs{u}\rightarrow \infty} \frac{|A(u)|}{\eta(u)} = 
0, \quad |F_\alpha(u)|\le C(1+\eta(u)), \qquad \alpha = 1,...,d,
\end{equation}
then (H5) follows directly from \eqref{H4thanos}, see Lemma~\ref{Lem:estimates} in the appendix.
\end{remark}
An analogue lemma under more restrictive assumptions
\begin{equation}\label{H4thanos-a}
\lim_{\abs{u}\rightarrow \infty} \frac{|A(u)|}{\eta(u)} = 
\lim_{\abs{u}\rightarrow \infty} \frac{|F_\alpha(u)|}{\eta(u)} = 
0, 
\end{equation}
was proved in \cite[Lemma A.1]{CleoTz}. Note however that \eqref{H4thanos-a} is not satisfied e.g. by compressible Euler equations. Any concentration in  terms $A$ and $F_\alpha$ are not present due to assumption \eqref{H4thanos-a}, which is a stronger requirement than (H4) assumed in the present paper. 
This however allowed the authors to omit the general representation of concentrations introduced in \cite{DiMa87} and \cite{AlBo}, because the concentration effect is considered just for the entropy, which is a non-negative  scalar function. Thus one can provide a simple derivation of weak limit as a Young measure and a concentration measure. Under slightly different assumptions on the entropy and in the same formulation as currently considered, i.e., $A(u)$ is not necessarily an identity, as in the aforementioned results, the issue of measure-valued-strong uniqueness was considered in \cite{CleoTz}. 

In the spirit of these results, the issue of mv-strong uniqueness was considered for various systems, including 
compressible Euler system and Savage-Hutter system describing granular media in \cite{GSW2015}, compressible Navier-Stokes in \cite{FGSW16} and complete compressible Euler system in \cite{BrFe}. An overview of these results is provided in \cite{tmna, Wiedemann}.
At this moment it is worth mentioning that the result  of B\v{r}ezina and Feireisl \cite{BrFe} does not fit in any of the presented frameworks for general hyperbolic systems, including also the framework presented in the current paper. Contrary to the other cases, they consider the full thermo-mechanical system. Thus a new element here is an appearance of the physical entropy. The system consisting of conservation of mass and conservation of momentum is not a closed system, as the pressure depends on the energy. To complete the system additional equation for the energy is considered. Then the role of an entropy $\eta$ should overtake a physical entropy, not as it was in the case of isentropic compressible Euler (as the system for the variables $\varrho, \vu$), when $\eta$ was the energy (kinetic and potential). In the setting of B\v{r}ezina and Feireisl the entropy inequality does not carry information that would allow to bound the flux $F_\alpha(u)$. We claim that appearance of thermal energy in the system results that the system does not fit into the approach initiated by Brenier et al.

The relative entropy method, which is fundamental for mv-strong uniqueness results, appears to be useful for other areas such as stability studies, asymptotic limits and dimension reduction problems (e.g. \cite{CleoTz}, \cite{GiTz}, \cite{FeJiNo}, \cite{BeFeNo}, \cite{BrKrMa}). 
Not only the systems describing phenomena of mathematical physics fall into these applications. Also results on problems arising from biology, cf. \cite{MiMiPe}, \cite{MiMiPe2}, \cite{Pe2007}, \cite{GwiWie}, can serve as examples.  The framework is known in this context as General Relative Entropy (GRE) and applies  for showing asymptotic convergence of solutions to steady-state solutions.
Finally we would like to underline how these results on  measure-valued solutions in fluid mechanics affected  certain numerical experiments, cf.~\cite{FMT16}.

\subsection{Dissipative measure-valued solutions}

Our interest is directed to the measure-valued-strong uniqueness principle for dissipative measure-valued solutions. We start with the motivation for our definition of measure valued solutions. 

Assume we have at hand a sequence of solutions $u^n$ solving some approximating problem
\begin{equation}\label{eq:app1}
\partial_tA(u^n)+\partial_\alpha F_\alpha(u^n)= P_n
\end{equation}
together with appropriate approximating entropy equation
\begin{equation}\label{eq:app2}
\partial_t\eta(u^n)+\partial_\alpha q_\alpha(u^n)= Q_n
\end{equation}
with $P_n,Q_n \rightarrow 0$ in appropriate topologies. Natural a priori bound for such problem is derived through the entropy equation \eqref{eq:app2} and yields 
\begin{equation}\label{eq:apriori}
\|\eta(u^n)\|_{L^\infty(0,T,L^1(\T^d))} \leq C.
\end{equation}
Due to our assumption (H4), see \eqref{H41}, we have the same $L^\infty(0,T,L^1(\T^d))$ bound for quantities $A(u^n)$ and $F_\alpha(u^n)$. Therefore due to Lemma~\ref{slicing} and Remark~\ref{rem-slicing} we are able to desintegrate concentration measures related to each of these quantities as follows
\begin{equation}\label{eq:desintegr}
m_f(dxdt) = m_f^t(dx) \otimes dt.
\end{equation}




Before defining solutions let us shortly describe the notation. By $\mathcal{P} \left(\overline X \right)$ we mean the set of probability measures on $\overline X$, $L^{\infty}_{\rm weak}\left( (0,T) \times \T^d; \mathcal{P} \left(\overline X\right) \right)$  stands for the space of weakly-star essentially bounded measurable maps with values in $\mathcal{P} \left(\overline X \right)$. We mean by  $\mathcal{M}({[0,T]\times\T^d})$  the space of measures on  ${[0,T]\times\T^d}$ and $ \mathcal{M}^+({[0,T]\times\T^d})$ refers to positive measures. 

\begin{definition}\label{def}
 We say that  $(\nu, m_A, m_{F_\alpha}, m_\eta)$, $\alpha = 1,...,d$, is a dissipative measure-valued solution of system~\eqref{system} with initial data $(\nu_{0,\cdot},m_A^0,m_\eta^0)$ if $\{ \nu_{t,x} \}_{(t,x) \in (0,T) \times \T^d }$, $\nu \in L^{\infty}_{\rm weak}\left( (0,T) \times \T^d; \mathcal{P} \left(\overline X \right) \right)$ is a parameterized measure 
and  together with concentration measures
$m_A\in (\mathcal{M}({[0,T]\times\T^d}))^n$,  $m_{F_{\alpha}}\in (\mathcal{M}({[0,T]\times\T^d}))^{n\times n}$ satisfy 
\begin{equation}\label{weak-form}
\begin{split}
\int_Q\langle \nu_{t,x}, A(\lambda)\rangle \cdot\partial_t \varphi dxdt +\int_Q \partial_t \varphi\cdot m_A(dxdt)+ \int_Q\langle \nu_{t,x}, F_\alpha(\lambda)\rangle\cdot\partial_\alpha \varphi dxdt\\
+\int_Q \partial_\alpha \varphi\cdot m_{F_{\alpha}}(dxdt)  
+\int_{\T^d}\langle \nu_{0,x}, A(\lambda)\rangle\cdot\varphi(0) dx+\int_{\T^d} \varphi(0)\cdot m_A^0(dx)=0
\end{split}
\end{equation}
for all $\varphi\in C^\infty_c(Q)^n$. 
Moreover, the total entropy balance holds for all $\zeta\in C^\infty_c([0,T))$
\begin{equation}\label{energy}
\begin{split}
\int_Q \langle \nu_{t,x}, \eta(\lambda)\rangle \zeta'(t) dxdt +\int_Q \zeta'(t)m_\eta(dxdt)+\int_{\T^d}\langle \nu_{0,x}, \eta(\lambda)\rangle\zeta(0) dx 
\\
+\int_{\T^d}\zeta(0)m_\eta^0(dx)\ge0
\end{split}
\end{equation}
with a dissipation measure $m_\eta\in  \mathcal{M}^+({[0,T]\times\T^d})$. 

\end{definition}

Throughout our paper we always assume that there exists a generating sequence of approximate solutions to the system \eqref{system}. Therefore we introduce the following definition.
\begin{definition}\label{gener}
 We say that  the dissipative measure-valued solution $(\nu, m_A, m_{F_{\alpha}}, m_\eta)$, $\alpha = 1,...,d$, of system~\eqref{system} is generated by a sequence of approximate solutions if there exists sequences $u^n$, $P_n$ and $Q_n$ such that \eqref{eq:app1}-\eqref{eq:app2} hold in the sense of distributions, $P_n$ and $Q_n$ converge to zero in distributions and 
 \[
f(u^n(t,x))dxdt \stackrel{*}{\rightharpoonup} \langle\nu_{t,x},f(\lambda)\rangle dxdt + m_f
\]
 hold for $f = A,F_\alpha$ and $\eta$.
\end{definition}

Our main theorem reads as follows.

\begin{theorem}\label{t:main}
Assume that hypothesis (H1)-(H5) hold. Let $(\nu, m_A, m_{F_{\alpha}}, m_\eta)$, $\alpha = 1,...,d$, be a dissipative measure-valued solution to~\eqref{system} generated by a sequence of approximate solutions. Let 
$U\in W^{1,\infty}(Q)$ be a strong solution to~\eqref{system} with the same initial data $u_0\in L^1(\R^d)$, thus 
$\nu_{0,x}=\delta_{u_0(x)}$, $m_A^0=m_\eta^0=0$. Then $\nu_{t,x}=\delta_{U(t,x)}$  a.e. in $Q$ and $m_A=m_{F_{\alpha}}=m_\eta=0$.
\end{theorem}

One of the key ingredients in the proof of Theorem \ref{t:main} is the following proposition stating relations between different concentration measures.
\begin{proposition}\label{p:meas}
Assume that the hypothesis (H1)-(H5) hold. Let $(\nu, m_A, m_{F_{\alpha}}, m_\eta)$, $\alpha = 1,...,d$, be a dissipative measure-valued solution to~\eqref{system} generated by a sequence of approximate solutions. Let $U\in W^{1,\infty}(Q)$ be a strong solution to~\eqref{system}. Then the dissipative measure valued solution $(\nu, m_A, m_{F_{\alpha}}, m_\eta)$ has the following properties:
\begin{enumerate}
\item[$(i)$] The concentration measure of the relative entropy $\eta(u|U)$ is equal to 
\[
m_\eta-m_A\cdot G(U)
\]
and 
\begin{equation}\label{positive1}
m_\eta-m_A\cdot G(U)\ge 0.
\end{equation}
\item[$(ii)$] The concentration measure of the relative  flux $F_\alpha(u|U)$ is equal to $$m_{F_{\alpha}} - \nabla F_\alpha(U)\nabla A(U)^{-1} m_A$$ and it is bounded by the concentration measure of the relative entropy, i.e.
\begin{equation}\label{positive2}
|m_{F_{\alpha}} - \nabla F_\alpha(U)\nabla A(U)^{-1} m_A| \leq C (m_\eta-m_A\cdot G(U)).
\end{equation}
\end{enumerate}
\end{proposition}

\subsection{Historical perspective}
Measure-valued solutions, despite being a relatively weak notion of solutions, play an important role in modern analysis of nonlinear systems of partial differential equations. The basic concept behind this approach is to embed the problem into a wider space. Instead of considering sequences solving approximate problems, which are some measurable functions, one passes to the level of parametrized  measures. The benefit of this idea is passing from a nonlinear problem to a linear one. The  essence of the proof of existence of such solutions becomes a matter of appropriate estimates rather than 
subtle weak sequential stability arguments. There is of course a cost to be paid -- the result of a limit is only a weak object represented by a {\it Young measure}, namely by a parametrized family of measures.  

 This framework begun with a celebrated paper of Young~\cite{young}, see also~\cite{ball} for a summary of the concept of Young measures. Later, Tartar~\cite{T79} and DiPerna~\cite{DiPerna} applied this approach to define measure-valued solutions   to scalar conservation laws
and, as a bystep in the proof of existence of entropy weak solutions, showed uniqueness of entropy measure-valued solutions (we mean by that solutions satisfying in addition a variant of entropy inequality for measures).


%
%
%
The next breakthrough is due to DiPerna and Majda who directed their attention to the incompressible Euler system. 
 Here, sequences of approximate solutions may not only oscillate, but also concentrate. Thus the original Young measure, capable of handling oscillations only, was insufficient to fully characterize weak limits of such sequences. An extension to {\it generalized Young measures} (or {\it DiPerna-Majda measures}) was later proposed, see~\cite{DiMa87} and also~\cite{AlBo} for some refinements. A measure-valued solution was then defined not    only as a  Young measure, but a triple describing oscillations, concentrations and concentration angle.   
%
Since this framework transfers to other systems and to general case as well we provide the full details in Section~\ref{ss:GYM}.

We direct our interest to measure-valued solutions to hyperbolic conservation laws. 
Unlike in the scalar case, for systems of conservation laws we cannot show uniqueness of entropy measure-valued solutions. The main obstacle to formulate analogous result is that, in most cases, we are equipped with only one entropy-entropy flux pair, contrary to a rich family of entropies available in the scalar case. Even more, the corresponding relative entropy inequality lacks appropriate symmetry. 

 For most systems of mathematical physics it is well known that even weak solutions may fail to be unique.
Only some conditional uniqueness can be claimed.
This conditional uniqueness property had been studied for many systems of fluid mechanics. First, in their classical papers, Prodi~\cite{Pro} and Serrin~\cite{Ser} had shown that a weak solution to the incompressible Navier-Stokes equations is unique and coincides with the strong solution, provided such a strong solution is known to exist. For conservation laws a conditional uniqueness of weak solutions was established firstly by Dafermos in~\cite{Da2000}. 
This is somehow an extension of the result on uniqueness of strong solutions (cf.~\cite{majda}), asserting that they are unique not only in the class of strong solutions, but also in the wider class of entropy weak solutions. 
This property became known as {\it weak-strong uniqueness}.


It was discovered, rather surprising, that the class of entropy weak solutions in the above can be widened to the class of measure-valued solutions which satisfy some kind of entropy inequality.
One can ask - {\it Is it to the benefit?} After all, measure-valued solutions seem a very weak notion and, admittedly, carry hardly any information about the physical problem. Nevertheless, measure-valued solutions, intimately related to Young measures, prove to be a powerful tool in the analysis of nonlinear PDEs. 

Numerous results on mv-strong uniqueness for various systems have already been described at the beginning of the introduction, as well as some of the results which concern a general hyperbolic case.

%


\section{Applications}\label{App}

In this section we provide a short list of applications of the general theory presented above. The first impression is that the general framework cannot cover e.g. incompressible Euler system. In Section~\ref{Extension} we show that a slight refinement allows to include not only incompressible Euler system, but also incompressible magnetohydrodynamics. 

\subsection{Compressible Euler system}\label{ss:CE}
The compressible Euler system is the following system of equations
\begin{align}
 \label{eq:CE1}
        \partial_t \rho+ \Div_x(\rho\vu) &= 0 \\ \label{eq:CE2}
        \partial_t (\rho\vu) + \Div_x(\rho\vu\otimes \vu) + \nabla_x p(\rho) &= 0,
\end{align}
for an unknown vector field $\vu\colon Q \to \R^n$
and scalar  $\rho\colon Q \to \R$. 
 The pressure $p(\rho)$ is a given function and if $p'(\rho) > 0$, the resulting system is a hyperbolic system of conservation laws. The associated entropy is given by
\begin{equation}
\eta(\rho,\vu) = \frac 12 \rho \abs{\vu}^2 + P(\rho),
\end{equation}
here the pressure potential $P(\rho)$ is related to the original pressure $p(\rho)$ through 
\begin{equation}\label{pr-potential}
P(\rho) = \rho \int_1^\rho\frac{p(r)}{r^2}dr.
\end{equation}
We assume the pressure satisfies the  following assumptions
\begin{equation} \label{BB6}
p \in C[0, \infty) \cap C^2(0, \infty), \ p(0) = 0, \ p'(\vr) > 0 \ \mbox{for}\ \vr > 0, \end{equation}
and
\begin{equation}
\liminf_{\vr \to \infty} p'(\vr) > 0,\
\liminf_{\vr \to \infty} \frac{P(\vr)}{p(\vr)} > 0.
\end{equation}

Since the quantity $\rho$ represents the physical density, we want it to be nonnegative, hence $X = (0,\infty) \times \R^n$ and $\overline{X} = [0,\infty) \times \R^n$.


We will show that the system satisfies the assumptions of Theorem~\ref{t:main} and fits into the presented framework.  We choose the variable $u$ to be $u=(u_1, u_2)=(\rho, \sqrt{\rho}\vu), u_1\in[0,\infty)$, $u_2\in\R^n$. Note that we have some freedom in choosing the variables, however keeping in mind that $A$ and $F_\alpha$ need to be continuously extendable from $X$ to $\overline X$. Our choice of variables is convenient
 for  further estimates. Note however that if $u=(\rho, m)$ with $m=\rho v$, then the second component of the flux $F_\alpha$ having then the form $\frac{m\otimes m}{\rho}$
does not extend continuously to $\overline X$. Nevertheless, in the chosen variables $u=(u_1, u_2)=(\rho, \sqrt{\rho}\vu)$ we have (denoting by $I_n$ the $n \times n$ identity matrix)
\begin{equation}
A(u)=\left(\begin{array}{c}u_1\\ \sqrt{u_1}u_2\end{array}\right),\quad F(u)=(F_1,\ldots,F_n)(u) = \left(\begin{array}{c}\sqrt{u_1}u_2\\ u_2\otimes u_2+p(u_1)I_n\end{array}\right).
\end{equation}
The entropy in these variables has  a form
\begin{equation}
\eta(u)=\frac{1}{2}|u_2|^2+P(u_1).
\end{equation}
Obviously hypothesis (H1) and (H2) are satisfied with
\begin{equation}
G(u)=\left(\begin{array}{c}P'(u_1)-\frac{1}{2u_1}|u_2|^2\\ \frac{u_2}{\sqrt{u_1}}\end{array}\right)
\end{equation}
and the matrix $\nabla^2\eta(u)-G(u)\cdot\nabla^2A(u)$ is equal to
\begin{equation}
\left(\begin{array}{cc}P''(u_1)+\frac{1}{4}\frac{|u_2|^2}{u_1^2} & -\frac{1}{2}\frac{u_2}{u_1} \\ -\frac{1}{2}\frac{u_2}{u_1} & I_n\end{array}\right)
\end{equation}
and is positive definite, hence (H3) is satisfied. Instead of checking hypothesis (H4)-(H5)  we will  check that \eqref{H4thanos} holds, see Remark~\ref{remark11}.

We want to show that 
\begin{equation}\label{A-eta}
\frac{|A(u)|}{\eta(u)}\to0
\end{equation}
as $|u|\to\infty$.
Observe firstly that 
\begin{equation}\label{Plog}
P(\rho)\ge \rho\log\rho
\end{equation}
for every $\rho\ge0$. Consider the convex functions 
$M_1(\rho):=\rho^2\sqrt{\log(\rho+1)}$ and  $M_2(\rho):=\rho^2$. It holds that 
\begin{equation}
\forall\lambda >0 \quad\lim\limits_{\rho\to\infty}\frac{M_2(\lambda \rho)}{M_1(\rho)}=0,
\end{equation}
indeed, as $\rho\to0$
\begin{equation}
\frac{\lambda^2 \rho^2}{\rho^2\sqrt{\log(\rho+1)}}=\frac{\lambda^2}{\sqrt{\log(\rho+1)}}\to0.
\end{equation}
This is equivalent to saying that the function $M_1$ is essentially stronger than $M_2$ (for the definition and the facts used in the sequel see Appendix \ref{s:convex}). Define  the Fenchel conjugate to $M$ as $M^*(\xi):=\sup_{\rho}(\xi\cdot\rho-M(\rho))$. Then the corresponding relation for  the conjugate functions reads as $M^*_2$ is essentially stronger than $M^*_1$ and as $M^*_2(\xi)=\xi^2$, then in particular
\begin{equation}\label{Mstar}
\lim\limits_{\xi\to\infty}\frac{M^*_1( \xi)}{M_2^*(\xi)}=\lim\limits_{\xi\to\infty}\frac{M^*_1( \xi)}{\xi^2}=0.
\end{equation}
The term $\sqrt{u_1}u_2$ is estimated with help of Fenchel-Young inequality as follows
\begin{equation*}
|\sqrt{u_1}u_2|\le M_1(\sqrt{u_1})+M^*_1(|u_2|).
\end{equation*}
This allows us to estimate 
\begin{equation}
\frac{|A(u)|}{\eta(u)}=\frac{u_1+|\sqrt{u_1}u_2|}{\frac{1}{2}|u_2|^2+P(u_1)}\le\frac{u_1}{P(u_1)}+
\frac{M_1(\sqrt{u_1})}{P(u_1)}+\frac{M_1^*(|u_2|)}{\frac{1}{2}|u_2|^2}.
\end{equation}
Taking into account \eqref{Plog} and  \eqref{Mstar} allows to conclude the above converges to zero as~$|u|\to\infty$. 
Moreover
\begin{equation}
\frac{|F(u)|}{\eta(u)}\le\frac{|\sqrt{u_1}{u_2}|+|u_2|^2+p(u_1)}{\frac{1}{2}|u_2|^2+P(u_1)}\le
\frac{\frac{1}{2}{u_1}+\frac{1}{2}|u_2|^2+|u_2|^2+p(u_1)}{\frac{1}{2}|u_2|^2+P(u_1)}
\end{equation}
thus the fraction is bounded and \eqref{H4thanos} is satisfied.
%
%
\begin{remark}
In the case $p(\rho) = \rho$ the pressure potential is given by $P(\rho) = \rho\log\rho$. In order to make the entropy $\eta(\rho,\vu)$ a nonnegative function, we have to add a proper constant, in this case the constant is $e^{-1}$, so we have
\[
\eta(\rho,\vu) = \frac 12 \rho \abs{\vu}^2 + \rho\log\rho + e^{-1}.
\]
Then the rest of the arguments follow the same lines. 
\end{remark}
\begin{remark}
Notice that condition~\eqref{A-eta} provides that the concentration measure related to a sequence $(A(u^n))$
will not appear. This can be immediately concluded from Proposition~\ref{p:gener}, which we prove later. Indeed, since \eqref{A-eta} provides that $$\lim\limits_{|u|\to\infty} |A(u)|\le C \lim\limits_{|u|\to\infty} \eta(u)$$ for any $C>0$, thus due to the Proposition~\ref{p:gener} $|m_A|\le C m_\eta$ also for any $C>0$ and hence $m_A\equiv 0$. 
\end{remark}

\subsection{Shallow water magnetohydrodynamics}

Consider the following system of equations of  shallow water magnetohydrodynamics
%
%
%

\begin{align}
    \label{eq:MHD_diver}
    \partial_t h+ \Div_x(h\vu) &= 0,\\
    \partial_t (h\vu) + \Div_x\left(h\vu\otimes \vu-h\vh\otimes \vh\right)+ \nabla_x (gh^2/2)&= 0,\\
    \partial_t (h\vh) + \Div_x(h\vh \otimes \vu-h\vu\otimes \vh) +\vu \Div_x(h\vh) &= 0,
\end{align}
where $g>0$ is the gravity constant, $h\colon Q\to\R_+$  is the thickness of the 
fluid, $\vu\colon Q\to\R^2$ is the velocity, $b\colon Q\to\R^2$ is the magnetic field. 
Note that once initially $\Div_x(h_0\vh_0)$ is zero, then $\Div_x(h\vh)$ vanishes for all times due to the transport equation for the magnetic field. Thus we can omit this term in further analysis. 

 We choose the  variables $u=(u_1, u_2, u_3)=(h, \sqrt{h}\vu, \sqrt{h} b)$, thus
\begin{equation}
A(u)=\left(\begin{array}{c}u_1\\ \sqrt{u_1}u_2\\\sqrt{u_1}u_3 \end{array}\right),\quad F(u) = (F_1,F_2)(u) =\left(\begin{array}{c}\sqrt{u_1}u_2\\ u_2\otimes u_2-u_3\otimes u_3+\frac{gu_1^2}{2}I_2\\u_3\otimes u_2-u_2\otimes u_3\end{array}\right)
\end{equation}
and the entropy 
\begin{equation}
\eta(u)=\frac{1}{2}|u_2|^2+\frac{1}{2}|u_3|^2+\frac{1}{2}gu_1^2.
\end{equation}
We observe that
\begin{equation}
G(u)=\left(\begin{array}{c}gu_1 - \frac{1}{2u_1}(|u_2|^2+|u_3|^2)\\ \frac{u_2}{\sqrt{u_1}}\\ \frac{u_3}{\sqrt{u_1}} \end{array}\right)
\end{equation}
and
\begin{equation}
\nabla^2\eta(u)-G(u)\cdot\nabla^2A(u) = \left(\begin{array}{ccc} g+\frac{|u_2|^2+|u_3|^2}{4u_1^2} & -\frac{1}{2}\frac{u_2}{u_1} & -\frac{1}{2}\frac{u_3}{u_1} \\ -\frac{1}{2}\frac{u_2}{u_1} & I_2 & 0 \\ -\frac{1}{2}\frac{u_3}{u_1} & 0 & I_2 \end{array}\right)
\end{equation}
and thus (H1)-(H3) are satisfied.

The appropriate estimates providing that \eqref{H4thanos} is satisfied follow the same lines as for compressible Euler system, the additional terms do not require any  new effort.


\subsection{Polyconvex elasticity}

In this section we consider the {\it system of elasticity}
\begin{equation}
\label{mainI}
\frac{\partial^2 y}{\partial t^2}=\nabla\cdot S(\nabla y),
\end{equation}
where $y \; : \; {Q} \times {\R}^+ \to{\R}^3$ stands for the motion, ${\mathbb F} = \nabla y$, $v = \partial_t y$,
and $S$ stands for the Piola-Kirchoff stress tensor
obtained as the gradient of a stored energy function, 
$S = \frac{\partial W}{\partial {\mathbb F}}$.  Here we assume that
$W$ is polyconvex, that is  $W({\mathbb F}) = G ( \Phi({\mathbb F}))$ where 
$G:{\mathbb M}^{3\times 3}\times{\mathbb M}^{3\times 3}\times \R \to [0,\infty)$
is a strictly convex function and $\Phi({\mathbb F}) = ({\mathbb F} ,\cof {\mathbb F}, \det {\mathbb F})\in
{\mathbb M}^{3\times 3}\times{\mathbb M}^{3\times 3}\times \R$
stands for the vector of null-Lagrangians: ${\mathbb F}$, the cofactor matrix $\cof {\mathbb F}$
and the determinant $\det {\mathbb F}$.
It is observed in~\cite{Dafermos1985} and~\cite{tzavaras1}
that this system can be embedded into the following symmetrizable hyperbolic system in a new dependent variable $\Xi=({\mathbb F},Z,w)$ taking values in
${\mathbb M}^{3\times 3}\times {\mathbb M}^{3\times 3}\times\R$
\begin{equation}
\begin{aligned}
\frac{\partial v_i}{\partial t}&=\frac{\partial}{\partial x^\alpha}\left(\frac{\partial G}{\partial\Xi^A}(\Xi)\frac{\partial\Phi^A}{\partial
{\mathbb F}_{i\alpha}}({\mathbb F})\right),\\
\frac{\partial\Xi^A}{\partial t}&=\frac{\partial}{\partial
x^\alpha}\left(\frac{\partial\Phi^A}{\partial
{\mathbb F}_{i\alpha}}({\mathbb F})v_i\right).
\end{aligned}
\end{equation}
This system admits the following entropy-entropy flux pair
\begin{equation}
\begin{aligned}
\eta(v,{\mathbb F},Z,w)&=\frac{1}{2}|v|^2+G({\mathbb F},Z,w),\\
q_\alpha&=v_i\,\frac{\partial G}{\partial\Xi^A}(\Xi)\frac{\partial\Phi^A}{\partial {\mathbb F}_{i\alpha}}({\mathbb F}).
\end{aligned}
\end{equation}
A strong solution to ~\eqref{mainI} is a function $y\in W^{2,\infty}$. It automatically satisfies
\begin{equation}\label{elastodynentropy}
\partial_t \eta(y) + \partial_\alpha q_\alpha(y) = 0. 
\end{equation}
\noindent
Under the following additional growth assumptions on the function $G$:
\begin{itemize}
\item[(A1)] $G \in C^3({\mathbb M}^{3\times 3} \times {\mathbb M}^{3\times 3} \times \R ; [0,\infty))$ is a strictly convex 
function satisfying for some $C >0$ the bound $D^2 G \ge C > 0$,

\item[(A2)] $G({\mathbb F},Z,w) \ge c_1 ( |{\mathbb F}|^p +  |Z|^q +  |w|^r + 1)  - c_2$ where $p\in (4, \infty), \ \ q, r \in [2,\infty)$,

\item[(A3)] $G({\mathbb F},Z,w) \le c (  |F|^p +  |Z|^q + |w|^r +1)$,

\item[(A4)] $| \partial_{\mathbb F} G| + |\partial_Z G|^{\frac{p}{p-1}} + |\partial_w G|^{\frac{p}{p-2} } \le o(1) (  |{\mathbb F}|^p +  |Z|^q + |w|^r +1)$ \quad where $o(1) \to 0$ as $|\Xi| \to \infty$,
\end{itemize}
an existence of dissipative measure-valued solutions as well as a weak-strong uniqueness result are proven, cf. \cite{DeStTz2, tzavaras1}.
According to the discussion in the Introduction it is enough to show that conditions~\eqref{H4thanos-a} are satisfied and thus (H5) follows. 
$$A(u)=\left(\begin{array}{c}v\\ {\mathbb F}\end{array}\right), \quad F_\alpha(u)=\left(\begin{array}{c}\frac{\partial G}{\partial {\mathbb F}}(\Xi)
\frac{\partial {\mathbb F}}{\partial {\mathbb F}_{\alpha}}\\\frac{\partial {\mathbb F}}{\partial {\mathbb F}_{\alpha}}v\end{array}\right).$$
By condition (A2) we conclude that 
\begin{equation}
\lim\limits_{|u|\to\infty}\frac{|A(u)|}{\eta(u)}=0.
\end{equation}
The combination of conditions (A2) and (A4) provides that 
\begin{equation}
\lim\limits_{|u|\to\infty}\frac{|F_\alpha(u)|}{\eta(u)}=0.
\end{equation}
For the discussion on the remaining assumptions (H1)-(H3) we refer the reader to~\cite{CleoTz}.

\section{Relations between concentration measures}



Our aim in this section is to prove Proposition \ref{p:meas}. We provide two proofs, the first 
one works with the Radon-Nikodym derivatives of measures, whereas the second one
relates our concept of dissipative measure valued solutions to the framework of generalized Young measures and is in its core based on the slicing lemma for products of measures. In particular, in the second proof we have to assume that the modified recession functions (for definition see below) exist for nonlinear functions appearing in our problem. 


\subsection{Radon-Nikodym derivatives of concentration measures}

Let us assume that we have a sequence of functions $u^n(x) : Y \rightarrow X$, here $Y$ is an underlying physical space, in the applications $Y = [0,T] \times \T^d$, and $X \subset \R^N$.

We recall the definition of the concentration measure related to a nonnegative nonlinear function $f: Y \times X \rightarrow \R^+$. The concentration measure $m_f$ is a nonnegative Radon measure such that
\[
\langle m_f,\chi \rangle  := \lim_{k\rightarrow\infty} \lim_{n \rightarrow\infty} \int_{Y \cap \{f \geq k\}} f(y,u^n(y))\chi(y) dy
\]
for all $\chi \in C_c(Y)$, $\chi \geq 0$.

Let $h: Y \times X \rightarrow \R^+$ be a nonnegative function satisfying 
\[
h(y,u) \leq Cf(y,u)
\]
for all $y \in Y$ and all $u \in X$. Then it is easy to observe that
\[
\{y; Cf(y,u(y)) \geq k \} \supset \{y; h(y,u(y)) \geq k\}
\]
and therefore
\[
\langle m_h,\chi \rangle   \leq \lim_{k\rightarrow\infty} \lim_{n \rightarrow\infty} \int_{Y \cap \{f \geq k\}} h(y,u^n(y))\chi(y) dy.
\]

If $g(y,u)$ is not a nonnegative function, we can split it into its positive and negative part
\[
g(y,u) = g^+(y,u) - g^-(y,u),
\]
and
\[|g(y,u)| = g^+(y,u) + g^-(y,u),\]
where both $g^+$ and $g^-$ are nonnegative. 
Thus we have 
\begin{equation}
\abs{\langle m_g,\chi\rangle} = \abs{\langle m_{g^+},\chi\rangle  - \langle m_{g^-},\chi\rangle } \leq \langle m_{g^+},\chi\rangle  + \langle m_{g^-},\chi\rangle = \langle m_{\abs{g}},\chi\rangle .
\end{equation}
Finally, just using the same argument componentwise and using as a norm for vectors in $\R^N$ the $l^\infty$ norm, we get the same for vector-valued functions $g$ and thus vector-valued concentration measures $m_g$. In particular if we assume
\[
\abs{g(y,u)} \leq C f(y,u)
\]
for a nonnegative function $f$ and a vector-valued function $g$, we end up with
\[
\abs{\langle m_g,\chi\rangle }  \leq \lim_{k\rightarrow\infty} \lim_{n \rightarrow\infty} \int_{Y \cap \{f \geq k\}} \abs{g(y,u^n(y))}\chi(y) dy.
\]

Next, we recall the concept of the Radon-Nikodym derivative of measures. Let $\mu_1$ and $\mu_2$ be nonnegative Radon measures such that $\mu_2 << \mu_1$. Then there exists a function $D_{\mu_1}\mu_2(x) \in L^\infty(\mu_1)$ called a Radon-Nikodym derivative of $\mu_2$ with respect to $\mu_1$ such that
\begin{equation}
\mu_2(A) = \int_{A} D_{\mu_1}\mu_2(x) d\mu_1(x).
\end{equation}
Moreover one can characterize the Radon-Nikodym derivative as follows (see e.g. \cite{Evans})
\begin{equation}
D_{\mu_1}\mu_2(x) = \lim_{\varepsilon \rightarrow 0^+}\frac{\mu_2(B(x,\varepsilon))}{\mu_1(B(x,\varepsilon))}
\end{equation}
for $\mu_1-$a.e. $x$. Here $B(x,r)$ denotes as usual the ball with center $x$ and radius $r$.

The definition of the Radon-Nikodym derivative can be extended using the Hahn-Jordan theorem to signed measures and then componentwise to vector valued signed measures.

First, let us define the continuous extension of a characteristic function of the ball of radius $\ep > 0$ as follows:
Let $\kappa_\ep : \R^+ \rightarrow \R^+$ be defined as follows
\begin{align}
\kappa_\ep(x) &= 1 \qquad \text{ for } x \in (0,\ep) \\
\kappa_\ep(x) &= 2 - \frac{x}{\ep} \qquad \text{ for } x \in (\ep,2\ep) \\
\kappa_\ep(x) &= 0 \qquad \text{ for } x \in (2\ep,+\infty).
\end{align}
Fix $x \in Y$ and define $\chi_{x,\ep} : Y \rightarrow \R^+$ as $\chi_{x,\ep}(y) := \kappa_\ep(\abs{y-x})$.

\begin{proposition}\label{p:RNder}
Let $\mu_1$ and $\mu_2$ be nonnegative Radon measures such that $\mu_2 << \mu_1$. Then it holds
\begin{equation}\label{RNder}
D_{\mu_1}\mu_2(x) = \lim_{\ep \rightarrow 0^+}\frac{\langle \mu_2,\chi_{x,\ep}\rangle }{\langle \mu_1,\chi_{x,\ep}\rangle }.
\end{equation}
\end{proposition}

In order to prove Proposition \ref{p:RNder} we need two elementary observations. Firstly it is a matter of a simple computation to check that
\[
\chi_{x,\ep}(y) = \frac{1}{\ep}\int_\ep^{2\ep} \chi_{B(x,s)}(y)ds.
\]
Secondly we need the following lemma.
\begin{lemma}\label{l:ab}
Let $a,b$ be nonnegative functions and let
\[
\lim_{\ep\rightarrow 0^+} \frac{a(\ep)}{b(\ep)} = M.
\]
Then
\begin{equation}
\lim_{\ep\rightarrow 0^+} \frac{\int_\ep^{2\ep} a(s)ds}{\int_\ep^{2\ep} b(s)ds} = M.
\end{equation}
\end{lemma}
\begin{proof}
Directly from the assumption of the lemma we have
\[
\lim_{\ep\rightarrow 0^+} \sup_{s \in (\ep,2\ep)} \frac{a(s)}{b(s)} = M.
\]
This yields that for every $\delta > 0$ there exists $\ep_\delta > 0$ such that for all $\ep < \ep_\delta$ and all $s \in (\ep,2\ep)$ it holds
\[
(M-\delta)b(s) \leq a(s) \leq (M+\delta)b(s).
\]
Integrating this inequality we immediately get
\[
(M-\delta)\int_\ep^{2\ep}b(s) ds \leq \int_\ep^{2\ep}a(s) ds \leq (M+\delta)\int_\ep^{2\ep}b(s) ds,
\]
which concludes the proof.
\end{proof}
Using Lemma \ref{l:ab} we now prove Proposition \ref{p:RNder}. We have
\begin{align}
D_{\mu_1}\mu_2(x) &= \lim_{\ep \rightarrow 0^+} \frac{\langle \mu_2,\chi_{B(x,\ep)}\rangle }{\langle \mu_1,\chi_{B(x,\ep)}\rangle } \\ \nonumber
&= \lim_{\ep \rightarrow 0^+} \frac{\int_\ep^{2\ep}\langle \mu_2,\chi_{B(x,s)}\rangle ds}{\int_\ep^{2\ep}\langle \mu_1,\chi_{B(x,s)}\rangle ds} \\ \nonumber
&= \lim_{\ep \rightarrow 0^+} \frac{\int_\ep^{2\ep} \int_Y \chi_{B(x,s)}(y)d\mu_2(y) ds}{\int_\ep^{2\ep} \int_Y \chi_{B(x,s)}(y)d\mu_1(y) ds} \\ \nonumber
&= \lim_{\ep \rightarrow 0^+} \frac{\int_Y \int_\ep^{2\ep} \chi_{B(x,s)}(y)ds d\mu_2(y)}{\int_Y \int_\ep^{2\ep} \chi_{B(x,s)}(y)ds d\mu_1(y)} \\ \nonumber
&= \lim_{\ep \rightarrow 0^+} \frac{\int_Y \chi_{x,\ep}(y) d\mu_2(y)}{\int_Y \chi_{x,\ep}(y) d\mu_1(y)}.
\end{align}
The following proposition is a generalization of \cite[Lemma 2.1]{FGSW16}, however the proof follows differently, without using the connection between biting limit and  Young measures. 
\begin{proposition}\label{p:gener}
Let $f(y,u)$ be a nonnegative continuous function on $Y \times \overline X$ and let $g(y,u)$ be a vector-valued function, also   continuous  on $Y \times \overline X$ such that
\begin{equation}\label{propass}
\lim\limits_{|u|\to\infty}\abs{g(y,u)} \leq C \lim\limits_{|u|\to\infty}f(y,u).
\end{equation}
Let $m_f$ and $m_g$ denote the concentration measures related to $f$ and $g$ respectively. Then
\begin{equation}\label{mgmf}
\abs{m_g} \leq Cm_f, 
\end{equation}
i.e. $\abs{m_g}(A) \leq Cm_f(A)$ for any Borel set $A \subset Y$.
\end{proposition}

\begin{proof}
First we observe that $\abs{m_g} << m_f$ as a consequence of \eqref{propass}. Then for any Borel set $A \subset Y$ we have
\[
\abs{m_g}(A) = \int_A D_{m_f} \abs{m_g} dm_f \leq \|D_{m_f} \abs{m_g} \|_{L^\infty_{m_f}(A)}m_f(A)  \leq \|D_{m_f} \abs{m_g} \|_{L^\infty_{m_f}(Y)}m_f(A).
\]
However we also have
\begin{align}
D_{m_f} \abs{m_g}(x) &= \lim_{\ep \rightarrow 0^+} \frac{\langle \abs{m_g},\chi_{x,\ep}\rangle }{\langle m_f,\chi_{x,\ep}\rangle } \\ \nonumber
&\leq \lim_{\ep \rightarrow 0^+} \frac{\lim_{k\rightarrow+\infty}\lim_{n\rightarrow+\infty}\int_{Y\cap\{f \geq k\}}\abs{g(y,u^n(y))}\chi_{x,\ep}(y)dy}{\lim_{k\rightarrow+\infty}\lim_{n\rightarrow+\infty}\int_{Y\cap\{f \geq k\}}f(y,u^n(y))\chi_{x,\ep}(y)dy} \leq C.
\end{align}
\end{proof}

Finally we use Proposition \ref{p:gener} with $f = \eta$ and $g = A$ and $g = F_\alpha$ and then with $f = \eta(\cdot|U)$ and $g = F_\alpha(\cdot|U)$ to prove Proposition \ref{p:meas}.

\subsection{Generalized Young measures}\label{ss:GYM}

Let us recall here the result of \cite{AlBo} characterizing the weak limits of nonlinear functions applied to maps bounded in $L^p(\T^d)$. Suppose  $(u_n)_{n\in\mathbb{N}}$ is a sequence of maps bounded in $L^p(\T^d;\R^m)$. 
It was proved in~\cite{AlBo} that there exists a subsequence (not relabeled), a parametrized probability measure $\nu\in L_w^\infty(\T^d;\mathcal{P}(\R^m))$, a non-negative measure $m\in\mathcal{M}^+(\T^d)$, and a parametrized probability measure on a sphere $\nu^\infty\in L_w^\infty(\T^d,m;\mathcal{P}(\mathbb{S}^{m-1}))$ such that
\begin{equation}\label{GYM}
f(x,u_n(x))dx\stackrel{*}{\rightharpoonup}\int_{\R^m}f(x,\lambda)d\nu_x(\lambda)dx+\int_{\mathbb{S}^{m-1}}f^\infty_r(x,\beta)d\nu^\infty_x(\beta)m
\end{equation}
weakly-star in the sense of measures. Here, $f:\T^d\times\R^m\to\R$ is any Carath\'{e}odory function with well defined and continuous recession function $f^\infty_r: \T^d \times \mathbb{S}^{m-1} \rightarrow \R$ defined as
\begin{equation}\label{recession}
f^\infty_r(x,\beta):=\lim_{
s\rightarrow\infty}
\frac{f(x,s\beta')}{s^p}.
\end{equation} 

Note that the measure $\nu$ represents the classical Young measure describing the oscillations in the sequence, whereas the second term on the right hand side of \eqref{GYM} describes the concentrations.

We can easily observe that this framework does not apply e.g. in the case of isentropic compressible Euler system with a pressure given by $p(\rho)=\rho^\gamma$, with $\gamma\neq 2$.
Choosing the variables $\beta=(\beta_1,\beta_2)=(\rho,\sqrt{\rho}v)$ the flux function has a form
$f(\beta)=(\sqrt\beta_1\beta_2,\beta_1^\gamma+\beta_2\otimes\beta_2)$. Consider an approximate sequence
$u^n=(\rho^n, \sqrt{\rho^n}v^n)$. An entropy inequality provides a~priori bounds
\begin{equation}
\int_\Omega\left(\frac{1}{2}|\sqrt{\rho^n}v^n|^2+\frac{1}{\gamma-1}(\rho^n)^\gamma\right) dx \le c.
\end{equation}
Thus we cannot conclude there exists some $p$ that the sequence $u^n$ is uniformly bounded in $L^p$. Here the first component is bounded in $L^\gamma$ and the second one in $L^2$. In a consequence there is no possibility to define a recession function according to formula~\eqref{recession}. 

This example motivates us to claim 
that in many cases  the framework of Alibert and Bouchitt\'{e} needs a 
refinement  to allow for considering  sequences with components of different growth. Following~\cite{GSW2015} let us take a sequence  $u^n=(v^n,w^n)_k$ with $(v^n)$  bounded in $L^p(\Omega;\R^l)$ and $(w^n)$ bounded in $L^q(\Omega;\R^m)$ ($1\leq p,q<\infty$). Then we define the \emph{nonhomogeneous unit sphere} as follows
\begin{equation*}
\mathbb{S}^{l+m-1}_{p,q}:=\{(\beta_1,\beta_2)\in\R^{l+m}: |\beta_1|^{2p}+|\beta_2|^{2q}=1\}.
\end{equation*}

We can characterize the limit as in~\eqref{GYM} and this is valid for all integrands $f$ whose $p$-$q$-recession function exists and is continuous on $\bar{\Omega}\times\mathbb{S}^{l+m-1}_{p,q}$. The $p$-$q$-recession function is defined as
\begin{equation*}
f^\infty(x,\beta_1,\beta_2):=\lim_{s\rightarrow\infty}\frac{f(x',s^q\beta_1',s^p\beta_2')}{s^{pq}}.
\end{equation*} 
Such an approach however is one of possible frameworks. We could consider more general compactifications of $\R^n$ than compactification with a sphere.

Since  
\[
\lim_{s\rightarrow+\infty}\eta(su) = +\infty
\]
for all $u \in \mathbb{S}^{n-1} \cap X$ we would like to define modified recession function as follows.
Let $f(u): X \rightarrow \R$ be a smooth function and let $\eta(u)$ be an entropy related to  hyperbolic system \eqref{system}. Then the modified recession function $f^\infty(u) : \mathbb{S}^{n-1} \cap X \rightarrow \R$ reads as
\begin{equation}\label{recess}
f^\infty(u) = \lim_{s \rightarrow +\infty} \frac{f(su)}{\eta(su)}
\end{equation}
for any $u \in \mathbb{S}^{n-1} \cap X$.

However, again here such  defined  function may not necessarily be continuous. Thus the correct definition should rather have the form
\begin{equation}\label{finfty}
f^\infty(u)=\lim\limits_{s\to\infty}\frac{f(s^{\alpha_1}\beta_1, \ldots, s^{\alpha_n}\beta_n)}{\eta(s^{\alpha_1}\beta_1, \ldots, s^{\alpha_n}\beta_n)}
\end{equation}
with some $\alpha_i>0$, $i=1,\ldots,n$. 

Assuming that the modified recession functions $A^\infty$ and $F_\alpha^\infty$ are properly defined (according to the definition from \eqref{finfty}) it is easy to observe that the properties \eqref{H41}-\eqref{H42} from hypothesis (H4) transfer to
\begin{align}
|A^\infty(u)| &\leq C \label{H41rec} \\
|F_\alpha^\infty(u)| &\leq C, \qquad \alpha = 1,...,d. \label{H42rec}
\end{align}
Given a strong solution $U$ to  system \eqref{system} with values in a compact subset of $X$ we can also calculate the modified recession functions for the relative quantities $\eta(u|U)$ and $F_\alpha(u|U)$. We have
\begin{equation}\label{RErec}
\eta^\infty(u|U) = \lim_{s\rightarrow+\infty}\frac{\eta((s^{\alpha_1}u_1, \ldots, s^{\alpha_n}u_n)|U)}{\eta(s^{\alpha_1}u_1, \ldots, s^{\alpha_n}u_n)} = 1 - G(U)\cdot A^\infty(u)
\end{equation}
and
\begin{equation}\label{RFrec}
F_\alpha^\infty(u|U) = \lim_{s\rightarrow+\infty}\frac{F_\alpha((s^{\alpha_1}u_1, \ldots, s^{\alpha_n}u_n)|U)}{\eta(s^{\alpha_1}u_1, \ldots, s^{\alpha_n}u_n)}= F_\alpha^\infty(u) - \nabla F_\alpha(U)\nabla A(U)^{-1} A^\infty(u).
\end{equation}

Note that since both $\eta(u|U)$ and $\eta(u)$ are nonnegative, also the modified recession function $\eta^\infty(u|U)$ has the same property, i.e.
\begin{equation}\label{etainf}
\eta^\infty(u|U) = 1 - G(U)\cdot A^\infty(u) \geq 0
\end{equation}
for all $u \in S^{n-1} \cap X$. Moreover the upper bound for the relative flux \eqref{bound} is also transfered to the modified recession functions as
\begin{equation}\label{recbound}
\begin{split}
|F_\alpha^\infty(u|U)| &=  \lim_{s\rightarrow+\infty}\frac{|F_\alpha((s^{\alpha_1}u_1, \ldots, s^{\alpha_n}u_n)|U)|}{\eta(s^{\alpha_1}u_1, \ldots, s^{\alpha_n}u_n)} \\&
\leq C\lim_{s\rightarrow+\infty}\frac{\eta((s^{\alpha_1}u_1, \ldots, s^{\alpha_n}u_n)|U)}{\eta(s^{\alpha_1}u_1, \ldots, s^{\alpha_n}u_n)}= C\eta^\infty(u|U). 
\end{split}
\end{equation}

Now we use these bounds for the recession functions to derive the bounds for the concentration measures described in Proposition \ref{p:meas}. We have

\begin{align}
\eta(u^n|U)dxdt&\stackrel{*}{\rightharpoonup}\left<\nu_{t,x},\eta(\lambda|U)\right>dxdt+\int_{\mathbb{S}^{m-1}}\eta^\infty(\beta|U)d\nu^\infty_{t,x}(\beta)m_\eta \\ \nonumber
 &= \left<\nu_{t,x},\eta(\lambda|U)\right>dxdt+m_\eta - G(U)\cdot\int_{\mathbb{S}^{m-1}}A^\infty(\beta)d\nu^\infty_{t,x}(\beta)m_\eta \\ \nonumber
&= \left<\nu_{t,x},\eta(\lambda|U)\right>dxdt+m_\eta - G(U)\cdot m_A .
\end{align}

This proves the form of the concentration measure for $\eta(u|U)$ and also \eqref{positive1}, since $\eta^\infty(\beta|U)$ is a nonnegative function, $\nu^\infty_{t,x}$ is a probability measure (i.e. nonnegative) and $m_\eta$ is a nonnegative measure.

Our final aim is to prove \eqref{positive2}. We start with
\begin{align}
F_\alpha(u^n|U)dxdt&\stackrel{*}{\rightharpoonup}\left<\nu_{t,x},F_\alpha(\lambda|U)\right>dxdt+\int_{\mathbb{S}^{m-1}}F_\alpha^\infty(\beta|U)d\nu^\infty_{t,x}(\beta)m_\eta \\ \nonumber
 &= \left<\nu_{t,x},F_\alpha(\lambda|U)\right>dxdt \\ \nonumber
 & \qquad +\int_{\mathbb{S}^{m-1}}F_\alpha^\infty(\beta) - \nabla F_\alpha(U)\nabla A(U)^{-1} A^\infty(\beta) d\nu^\infty_{t,x}(\beta)m_\eta \\ \nonumber
&= \left<\nu_{t,x},F_\alpha(\lambda|U)\right>dxdt+ m_{F_{\alpha}} - \nabla F_\alpha(U)\nabla A(U)^{-1} m_A,
\end{align}
which proves the form of the concentration measure for the relative fluxes $F_\alpha(u|U)$. Finally we use \eqref{recbound} to argue that
\begin{align}
&| m_{F_{\alpha}} - \nabla F_\alpha(U)\nabla A(U)^{-1} m_A| = \int_{\mathbb{S}^{m-1}}|F_\alpha^\infty(\beta|U)|d\nu^\infty_{t,x}(\beta)m_\eta \\ \nonumber 
& \qquad \leq C \int_{\mathbb{S}^{m-1}}\eta^\infty(\beta|U)d\nu^\infty_{t,x}(\beta)m_\eta = C (m_\eta - G(U)\cdot m_A).
\end{align}
The proof of Proposition \ref{p:meas} is complete.

\section{Relative entropy inequality}

\subsection{Derivation of the relative entropy inequality}

We derive the relative entropy inequality. We choose in \eqref{weak-form} a test function $\varphi=\zeta(t)G(U(t,x))$ with $\zeta \in C^\infty_c([0,T))$. As $U$ is a strong solution, thus \eqref{system}  is satisfied by $U$, we multiply it with the same test function and integrate, finally to subtract it from \eqref{weak-form} to get
\begin{equation}
\begin{split}
\int_Q\zeta'(t)G(U)\cdot(\langle \nu_{t,x}, A(\lambda)\rangle-A(U))dxdt +\int_Q\zeta'(t)G(U)\cdot m_A(dxdt)  
\\
+ \int_Q\zeta(t)\partial_\alpha G(U)\cdot(\langle \nu_{t,x}, F_\alpha(\lambda)\rangle-F_\alpha(U))dxdt+ \int_Q\zeta(t)\partial_\alpha G(U)\cdot m_{F_\alpha}(dxdt)  
\\
+\int_Q\zeta(t)\partial_tG(U)\cdot(\langle \nu_{t,x}, A(\lambda)\rangle-A(U))dxdt +\int_Q\zeta(t)\partial_tG(U)\cdot m_A(dxdt) 
\\
+\int_{\T^d}\zeta(0)G(U(0))\cdot(\langle \nu_{0,x}, A(\lambda)\rangle-A(U(0)))dx+\int_{\T^d}\zeta(0)G(U(0))\cdot m_A^0(dx) =0.
\end{split}
\end{equation}
Following \cite{tzavaras1} we define the averaged quantities
\begin{equation}\label{dfH}
\mathcal{H}({\nu},U)\doteq\langle \boldsymbol{\nu},\eta\rangle-\eta(U)-G(U)\cdot \big ( \langle\boldsymbol{\nu},A\rangle-A(U) \big ),
\end{equation}
\begin{equation}
\label{dfZ}
Z_\alpha({\nu},U)\doteq\langle \boldsymbol{\nu},F_\alpha\rangle-F_\alpha(U)-\nabla F_\alpha(U)\nabla A(U)^{-1}(\langle\boldsymbol{\nu},A\rangle-A(U))\;.
\end{equation}
Since $\partial_tU=-(\nabla A(U))^{-1}\nabla F_\alpha(U)\partial_\alpha U$ we observe the following 
\begin{equation*}
\begin{split}
\partial_tG(U)\cdot &(\langle \nu_{t,x}, A(\lambda)\rangle-A(U))+ \partial_\alpha G(U)\cdot(\langle \nu_{t,x}, F_\alpha(\lambda)\rangle-F_\alpha(U))
\\
&=\nabla G(U)\partial_t U\cdot (\langle \nu_{t,x}, A(\lambda)\rangle-A(U))+\nabla G(U)\partial_\alpha U\cdot (\langle \nu_{t,x}, F_\alpha(\lambda)\rangle-F_\alpha(U))
\\
&=\nabla G(U)(\nabla A(U))^{-1}\nabla F_\alpha(U)\partial_\alpha U\cdot (\langle \nu_{t,x}, A(\lambda)\rangle-A(U))
\\
&+
\nabla G(U)\partial_\alpha U\cdot (\langle \nu_{t,x}, F_\alpha(\lambda)\rangle-F_\alpha(U)
\\
&=:\nabla G(U)\partial_\alpha U\cdot Z_\alpha(\nu_{t,x},U).
\end{split}
\end{equation*}

Using the entropy inequality for measure-valued solutions \eqref{energy}
we obtain
\begin{equation}\label{estimates20}
\begin{split}
\int_Q\zeta'(t)& \mathcal{H}(\nu,U)+\int_Q\zeta'(t)(m_\eta-m_A\cdot G(U))(dxdt)
\\
&\ge\int_Q \zeta(t)\nabla G(U) \partial_\alpha U\cdot Z_\alpha(\nu, U) \\&+
\int_Q\zeta(t)(m_A\cdot\nabla G(U)\partial_tU+m_{F_{\alpha}}\cdot\nabla G(U)\partial_\alpha U)(dxdt)
\\
&-\int_{\T^d}\zeta(0)(\langle\nu_{0,x},\eta\rangle-\eta(U(0))-(\langle\nu_{0,x}, A\rangle-A(U(0)))\cdot G(U(0)) dx
\\
&-\int_{\T^d}\zeta(0)(m_\eta^0-m_A^0\cdot G(U(0))) (dx).
\end{split}
\end{equation}
Using \eqref{H23}-\eqref{H24} we compute
\begin{equation}
\begin{split}
 m_A\cdot&\nabla G(U)\partial_tU+m_{F_{\alpha}}\cdot\nabla G(U)\partial_\alpha U
\\
&=-m_A\cdot\nabla G(U)(\nabla A(U))^{-1}\nabla F_\alpha(U)\partial_\alpha U+m_{F_{\alpha}}\cdot\nabla G(U)\partial_\alpha  U
\\
&=-m_A\cdot \nabla A^{-T}\nabla G^T\nabla F_\alpha\partial_\alpha U+m_{F_{\alpha}}\cdot\nabla G(U)\partial_\alpha U
\\
&=-m_A\cdot\nabla A^{-T}\nabla F_\alpha\nabla G(U)\partial_\alpha U+m_{F_{\alpha}}\cdot \nabla G\partial_{\alpha} U
\\
&=(-\nabla F_\alpha(U)\nabla A(U)^{-1}m_A+m_{F_{\alpha}})\cdot \nabla G(U)\partial_\alpha U.
\end{split}
\end{equation}
In a standard way we choose $\zeta=\zeta^n$ to be a sequence of smooth monotone functions which approximate the characteristic function of the interval $[0,\tau]$ and pass to the limit, thus \eqref{estimates20} turns into
\begin{equation}
\begin{split}
\int_{\T^d}\mathcal {H} (\nu, U)(\tau)dx &+\int_{\T^d}(m_\eta^\tau-m_A^\tau\cdot G(U(\tau))) (dx)
\le C(U)\int _0^\tau\int_{\T^d}\max_\alpha |Z_\alpha| dxdt\\
&+\int _0^\tau\int_{\T^d}(-\nabla F_\alpha(U)(\nabla A(U))^{-1}m_A+m_{F_{\alpha}})\cdot \nabla G(U)\partial_\alpha U dxdt
\\
&+\int_{\T^d}\mathcal {H} (\nu, U)(0)dx +\int_{\T^d}(m_\eta^0-m_A^0\cdot G(U(0))) (dx).
\end{split}
\end{equation}
Note that in a same way as in \cite{tzavaras1} we have
\begin{equation}
\begin{aligned}
Z_\alpha(\nu,U)  &= 
 \langle {\nu},F_\alpha\rangle-F_\alpha(U)-\nabla F_\alpha(U)\nabla A(U)^{-1}(\langle{\nu},A\rangle-A(U))
 \\
 &= \langle {\nu} , F_\alpha (\lambda | U ) \rangle
 \\
 &\le  C_1 \langle {\nu} , \eta (\lambda | U ) \rangle 
 \\
 &= C_1 \mathcal{H} ( {\nu}, U).
\end{aligned}
\end{equation}
Consequently, using \eqref{positive2} the estimate which allows us to use the Gronwall inequality has a form
\begin{equation}
\begin{split}
\int_{\T^d}\mathcal {H}(\nu, U)(\tau) dx +\int_{\T^d}(m_\eta^\tau-m_A^\tau\cdot G(U(\tau)))(dx)\le
\\
C\left(\int_0^\tau\int_{\T^d}\mathcal {H}(\nu,U) dxdt + \int_0^\tau\left(\int_{\T^d} (m_\eta^t-m_A^t\cdot G(U(t)))(dx)\right)dt \right)\\
+ \int_{\T^d} \langle\nu_{0,x},\eta(\lambda|U(0))\rangle dx + \int_{\T^d}(m_\eta^0-m_A^0\cdot G(U(0)))(dx).
 \end{split}
\end{equation}
Thus, 
\begin{align}
&\int_{\T^d}\mathcal {H}(\nu, U)(t) dx + \int_{\T^d}(m_\eta^t-m_A^t\cdot G(U(t)))(dx) \\\nonumber
& \qquad \le C\left(\int_{\T^d} \langle\nu_{0,x},\eta(\lambda|U(0))\rangle dx + \int_{\T^d}(m_\eta^0-m_A^0\cdot G(U(0)))(dx)\right)\ e^{ct}.
\end{align}
In particular, we observe that if $\nu_{0,x} = \delta_{U(0,x)}$ and $m_\eta^0 = m_A^0 = 0$, then $\nu_{t,x} = \delta_{U(t,x)}$ a.e. and 
\begin{equation}\label{eq:blbl}
m_\eta^t - G(U(t))\cdot m_A^t = 0
\end{equation}
for almost any $t$. Note that at this point it is enough to have $m_\eta^0-m_A^0\cdot G(U(0)) = 0$ to reach the same conclusion. What remains now is to show that the concentration measures $m_\eta$, $m_A$ and $m_{F_\alpha}$ are all equal to zero. This is done by comparing the definition of the measure valued solution with $\nu_{t,x} = \delta_{U(t,x)}$ which we already know with the fact that $U(t,x)$ is a (strong) solution to the system \eqref{system}. Using here also the fact that $m_A^0 = 0$ we obtain
\begin{equation}
\int_0^T\int_{\T^d} \partial_t\varphi\cdot m_A(dxdt) + \int_0^T \int_{\T^d}\partial_\alpha \varphi\cdot m_{F_\alpha}(dxdt) = 0
\end{equation}
for all $\varphi \in C^\infty_c(Q)$. This yields $m_A = 0$ and $m_{F_\alpha} = 0$ and thus consequently $m_\eta = 0$ due to~\eqref{eq:blbl}.

\section{Extension}\label{Extension}

As one may easily observe unfortunately this general framework will not cover systems of conservation laws, which may fail to be hyperbolic, typically incompressible inviscid systems. 
In the current approach we present a simple extension of the presented framework to cover the case of incompressible fluids, in case of which the assumption that $\nabla A$ is a nonsingular matrix is not satisfied. For this reason we distinguish from the flux the part $L$ (Lagrange multiplier) which is perpendicular to the vector $G(U)$ (which coincides with the gradient of the entropy of the strong solution in the case $A = \text{Id}$). Thus we assume that there exists a subspace $Y$, such that $G(U)\in Y$  and $L\in Y^\bot$, where $U$ is a strong solution to the considered system.
Let us then consider a system in the following form
\begin{equation}\label{system2}
\partial_t A(u)+\partial_\alpha F_\alpha(u)+L=0.
\end{equation}
The fact that $G(U)\in Y$ can help a lot in constructing entropies for the system \eqref{system2}. Therefore we reformulate the hypothesis (H2) as follows.
\begin{itemize}
\item[(H2')] The system \eqref{system2} is endowed with a companion law
\begin{equation}\label{entropy2}
\partial_t \eta(u) + \partial_\alpha q_\alpha(u) + \overline{L}_\alpha\cdot\partial_\alpha u = 0
\end{equation}
with an entropy $\eta: \overline X\to \R_+$, 
such that $\eta(u) \geq 0$ 
and 
\begin{equation}\label{eta-infty2}
\lim\limits_{|u|\to\infty}\eta(u)=\infty.
\end{equation}
This means we assume the existence of a smooth function $G: X \rightarrow \R^n$ such that 
\begin{equation}
\nabla \eta = G \cdot \nabla A \label{H211}
\end{equation}
and the condition \eqref{H22} is relaxed the following way. We assume that
\begin{equation}
G \cdot \nabla F_\alpha = \nabla q_\alpha + \overline{L}_\alpha, \qquad \alpha = 1,...,d. \label{H221}
\end{equation}
with the additional property that $\overline{L}_\alpha \cdot \partial_\alpha u = 0$ for all $u$ such that $G(u) \in Y$. 

\end{itemize}

We need a slight modification of the definition of measure valued solution, namely the class of test functions will change. 
\begin{definition}\label{def2}
 We say that  $(\nu, m_A, m_{F_\alpha}, m_\eta)$, $\alpha = 1,...,d$, is a dissipative measure-valued solution of system~\eqref{system2} with initial data $(\nu_{0,\cdot},m_A^0,m_\eta^0)$ if $\{ \nu_{t,x} \}_{(t,x) \in (0,T) \times \T^d }$, $\nu \in L^{\infty}_{\rm weak}\left( (0,T) \times \T^d; \mathcal{P} \left(\overline X \right) \right)$ is a parameterized measure 
and  together with concentration measures
$m_A\in (\mathcal{M}({[0,T]\times\T^d}))^n$, $m_{F_{\alpha}}\in (\mathcal{M}({[0,T]\times\T^d}))^{n\times n}$ satisfy 
\begin{equation}\label{weak-form2}
\begin{split}
\int_Q\langle \nu_{t,x}, A(\lambda)\rangle \cdot\partial_t \varphi dxdt +\int_Q \partial_t \varphi\cdot m_A(dxdt)+ \int_Q\langle \nu_{t,x}, F_\alpha(\lambda)\rangle\cdot\partial_\alpha \varphi dxdt\\
+\int_Q \partial_\alpha \varphi\cdot m_{F_{\alpha}}(dxdt)  
+\int_{\T^d}\langle \nu_{0,x}, A(\lambda)\rangle\cdot\varphi(0) dx+\int_{\T^d} \varphi(0)\cdot m_A^0(dx)=0
\end{split}
\end{equation}
for all $\varphi\in (C^\infty_c(Q)\cap Y)^n$. 
Moreover, the total entropy balance holds for all nonnegative $\zeta\in C^\infty_c([0,T))$ 
\begin{equation}\label{energy2}
\begin{split}
\int_Q \langle \nu_{t,x}, \eta(\lambda)\rangle \zeta'(t) dxdt +\int_Q \zeta'(t)m_\eta(dxdt)+\int_{\T^d}\langle \nu_{0,x}, \eta(\lambda)\rangle\zeta(0) dx 
\\
+\int_{\T^d}\zeta(0)m_\eta^0(dx)\ge0
\end{split}
\end{equation}
with a dissipation measure $m_\eta\in  \mathcal{M}^+({[0,T]\times\T^d})$. 

\end{definition}

Then an analogue result on mv-strong uniqueness in this case requires to add the constrain on strong solution, which allows to use the 
vector $G(U)$ as a test function in a distributional formulation. Thus accordingly we require that the strong solution is such that $G(U)$ belongs to the subspace $Y$. For this purpose we define a space $W^{1,\infty}_{Y}(Q)$ as the set of all elements $\psi$, which are in $W^{1,\infty}(Q)$ and $G(\psi) \in Y$. 

\begin{theorem}\label{t:main-1}
Assume that the hypothesis (H1)-(H5) hold with (H2) replaced by (H2'). Let $(\nu, m_A, m_{F_{\alpha}}, m_\eta)$, $\alpha = 1,...,d$, be a dissipative measure-valued solution to~\eqref{system2} generated by a sequence of approximate solutions. Let 
$U\in W^{1,\infty}_{Y}(Q)$ be a strong solution to~\eqref{system2} with the same initial data $u_0\in L^1(\R^d)$, thus 
$\nu_{0,x}=\delta_{u_0(x)}$, $m_A^0=m_\eta^0=0$. Then $\nu_{t,x}=\delta_{U(t,x)}$  a.e. in $Q$ and $m_A=m_{F_{\alpha}}=m_\eta=0$.
\end{theorem}

\bigskip

\subsection{Examples for the extended case}\label{ss:Extension}
\subsubsection{Incompressible Euler system}
Consider the system 
\begin{align}
 \label{eq:E1}
        \partial_t \vu + \Div_x(\vu\otimes \vu) + \nabla_x p &= 0,\\ \label{eq:E2}
          \Div_x\vu &= 0, 
\end{align}
for an unknown vector field $\vu\colon Q \to \R^n$
and scalar  $p\colon Q \to \R$. 

In this case we forget in the formulation about the divergence free constrain as this information will be carried in the definition of the space $Y$, namely $u=\vu$, $A\equiv Id$, $F(u)=\vu\otimes \vu$ and $L$ corresponds to the gradient of the pressure. The entropy $\eta = \frac 12 |\vu|^2$ and $G = \nabla \eta = \vu$. We see that the space $Y$ is the space of divergence free smooth vector fields $\vu$. 

A direct calculation yields $q_\alpha = \frac 12 \vu_\alpha |\vu|^2$ and therefore $\overline{L}_\alpha = \frac 12 |\vu|^2 e_\alpha$, where $e_\alpha$ is the unit vector in the $\alpha$ direction. We easily check that $\overline{L}_\alpha\cdot \partial_\alpha u = \frac 12 |\vu|^2 \Div_x \vu = 0$ if $G(u) = \vu \in Y$.

%
%
%

\subsubsection{Incompressible magnetohydrodynamics}
Let us consider the system
\begin{align}
           \partial_t \vu + \Div_x(\vu\otimes  \vu - b\otimes b) + \nabla_x (p+\frac{1}{2}|b|^2) &= 0
            \\
            \partial_t b + \Div_x(\vu\otimes b-b \otimes \vu)  &= 0,\\
              \Div_x\vu &= 0
            \\
            \Div_xb&= 0 
\end{align} 
for unknown vector functions $\vu\colon Q \to \R^n$ and $b\colon Q \to \R^n$ and an unknown scalar function $p\colon Q\to \R$. It is sufficient to require that $\Div_x b$ is equal to zero at the initial time as the information is then transported. The system  describes the motion of an ideal electrically conducting fluid, see e.g.~\cite[Chapter VIII]{landau}.

Here $u=(\vu,b)^T$, $A\equiv Id$, $L=(\nabla_x (p+\frac{1}{2}|b|^2),0)^T$ and 
$$F(v,b)=(F_1,...,F_n)(v,b) = \left(\begin{array}{c}v\otimes v-b\otimes b\\v\otimes b-b \otimes v\end{array}\right).$$
Similarly as in the previous case, $\eta = \frac12 (|\vu|^2 + |b|^2)$ and $G = \nabla \eta = u = (\vu,b)^T$.
The space $Y$ is the space of divergence free smooth vector fields in the first (velocity) variable, a common feature in the incompressible problems. 

For this system the entropy fluxes are $q_\alpha = \frac 12(|\vu|^2 + |b|^2)\vu_\alpha - (\vu\cdot b) b_\alpha$, consequently we derive that $$\overline{L}_\alpha = \left(\begin{array}{c} \frac12 (|\vu|^2 + |b|^2) e_\alpha \\ -(\vu\cdot b) e_\alpha \end{array}\right).$$ 
Hence $\overline{L}_\alpha \cdot \partial_\alpha u = \frac12 (|\vu|^2 + |b|^2)\Div_x\vu - (\vu\cdot b)\Div_x b$ and we see that the first term is zero whenever $G(u) \in Y$ and the second term is zero whenever $\Div_x b = 0$ at the initial time $t = 0$.

\subsubsection{Nonhomogeneous incompressible Euler system}
Here we consider the system 
\begin{align}
 \label{eq:NE1}
\partial_t \rho + \Div_x(\rho\vu) &=0, \\ \label{eq:NE2}
\partial_t (\rho\vu) + \Div_x(\rho\vu\otimes \vu) + \nabla_x p &= 0,\\ \label{eq:NE3}
          \Div_x\vu &= 0, 
\end{align}
for an unknown vector field $\vu\colon Q \to \R^n$ and scalar fields $\rho\colon Q \to \R_+$ and $p\colon Q \to \R$. Similarly as in the compressible Euler example \ref{ss:CE} we choose the state variables to be $u = (u_1,u_2)^T = (\rho,\sqrt{\rho}\vu)^T$. Then $A(u) = (u_1,\sqrt{u_1}u_2)^T$, $L=(0,\nabla_x p)^T$ and 
\begin{equation*}
F(u)=(F_1,\ldots,F_n)(u) = \left(\begin{array}{c}\sqrt{u_1}u_2\\ u_2\otimes u_2\end{array}\right).
\end{equation*}
The divergence-free condition \eqref{eq:NE3} allows us to choose as the entropy the function
\begin{equation*}
\eta(u)=\frac{1}{2}(|u_2|^2+u_1^2).
\end{equation*}
Indeed, in this case we get 
$$G(u) = \left(\begin{array}{c}u_1 - \frac{|u_2|^2}{2u_1} \\ \frac{u_2}{\sqrt{u_1}}\end{array}\right)$$
and the space $Y$ will be again the space of smooth functions $u$ such that the divergence of the second component of $G(u)$ is equal to zero, i.e. the space of states with divergence-free velocities.

The entropy fluxes are now $q_\alpha = \frac 12 (u_1^2 + |u_2|^2)\frac{u_{2\alpha}}{\sqrt{u_1}}$ and we derive that 
$$\overline{L}_\alpha(u) = \left(\begin{array}{c}-\frac 14 \sqrt{u_1}u_{2\alpha} \\ \frac 12 u_1^{3/2}e_\alpha\end{array}\right).$$
A straightforward computation reveals that $\overline{L}_\alpha\cdot \partial_\alpha u = \frac 12 u_1^2 \Div_x \frac{u_2}{\sqrt{u_1}}$ and thus it equals to zero whenever $u$ is such that $G(u) \in Y$.

\subsubsection{Nonhomogeneous incompressible magnetohydrodynamics}
We consider the system
\begin{align}
 \label{eq:NEMHD1}
\partial_t \rho + \Div_x(\rho\vu) &=0, \\ \label{eq:NEMHD2}
\partial_t (\rho\vu) + \Div_x(\rho\vu\otimes \vu - b\otimes b) + \nabla_x (p + \frac12 |b|^2) &= 0,\\ \label{eq:NEMHD3}
\partial_t b + \Div_x(v\otimes b - b \otimes v) &= 0, \\ \label{eq:NEMHD4}
          \Div_x\vu &= 0, \\ \label{eq:NEMHD5}
					\Div_x b &= 0,
\end{align}
with unknown vector fields $\vu\colon Q \to \R^n$ and $b\colon Q \to \R^n$ and scalar fields $\rho\colon Q \to \R_+$ and $p\colon Q \to \R$. Again if we assume $\Div_x b = 0$ at the time $t = 0$, this information gets transported.

In order to be able to continuously extend the fluxes $F_\alpha$ for zero densities, we can not proceed the same way as in the previous cases. Therefore we now choose the state variables to be $u = (u_1,u_2,u_3)^T = (\rho,\vu,b)^T$. Then $A(u) = (u_1,u_1u_2,u_3)^T$, $L=(0,\nabla_x (p + \frac 12 |b|^2,0)^T$ and 
\begin{equation*}
F(u)=(F_1,\ldots,F_n)(u) = \left(\begin{array}{c}u_1u_2\\ u_1u_2\otimes u_2 - u_3 \otimes u_3 \\ u_2\otimes u_3 - u_3 \otimes u_2\end{array}\right).
\end{equation*}
Similarly as in the case of nonhomogeneous Euler equations, we can choose as the entropy the function
\begin{equation*}
\eta(u)=\frac{1}{2}(u_1^2 + u_1|u_2|^2+|u_3|^2).
\end{equation*}
and obtain
$$G(u) = \left(\begin{array}{c}u_1 - \frac12 |u_2|^2 \\ u_2 \\ u_3 \end{array}\right).$$
The space $Y$ will be once again the space of smooth functions $u$ such that the divergence of the second component of $G(u)$ is equal to zero, i.e. the space of states with divergence-free velocities.

The entropy fluxes are $q_\alpha = \frac 12 (u_1^2 + u_1|u_2|^2 + |u_3|^2)u_{2\alpha} - u_2\cdot u_3 u_{3\alpha}$ and consequently
$$\overline{L}_\alpha(u) = \left(\begin{array}{c}0 \\ \frac 12 (u_1^2 + |u_3|^2) e_\alpha \\ -u_2\cdot u_3 e_\alpha \end{array}\right).$$
In particular we see that $\overline{L}_\alpha\cdot \partial_\alpha u = \frac 12 (u_1^2 + |u_3|^2) \Div_x u_2 - u_2\cdot u_3 \Div_x u_3$. Similarly as in the example of incompressible magnetohydrodynamics we conclude that $\overline{L}_\alpha\cdot \partial_\alpha u = 0$ for all $u$ such that $G(u) \in Y$ if we moreover assume that $\Div_x b = 0$ at the initial time $t = 0$.

\appendix

\section{Auxilary facts}
We include a lemma similar to~\cite[Lemma A.1]{CleoTz}, however under weaker assumptions, see the discussion in Remark~\ref{remark}. The proof follows similar lines, however we include it for reader's convenience. 
\begin{lemma}\label{Lem:estimates}
Let (H1)-(H3) and \eqref{H4thanos}  be satisfied. Then for each $\alpha=1,\ldots, d$ 
\begin{equation}\label{rel-est-appendix}
F_\alpha(u|U)\le C\eta(u|U)
\end{equation}
for each $u\in \overline X$ and bounded $U$.
\end{lemma}
\begin{proof}
Let $D\subset X$ be a compact set and let $D_\varepsilon:=\{y+\alpha: y\in D, |\alpha|<\varepsilon\}$.
The condition (H3) does not provide that $\eta$ is necessarily convex, nevertheless we can introduce an entropy $H=\eta\circ A^{-1}$ which is already uniformly convex in $D$.  If we define now
\begin{equation}
H(A(u)|A(U)):=H(A(u))-H(A(U))-\nabla_vH(A(U))(A(u)-A(U))
\end{equation}
we immediately observe that 
\begin{equation}
\eta(u|U)=\eta(u)-\eta(U)-G(U)(A(u)-A(U))=H(A(u)|A(U))
\end{equation}
just due to chain rule and an observation that $\nabla \eta(u)=\nabla_vH(A(u))\nabla A(u)$. We also introduce a flux in new variables, i.e., $Q_\alpha=F_\alpha\circ A^{-1}$. Similarly we  define the relative flux of~$Q_\alpha$
\begin{equation}
Q_\alpha(A(u)|A(U)):=Q_\alpha(A(u))-Q_\alpha(A(U))-\nabla_vQ_\alpha(A(U))(A(u)-A(U)).
\end{equation}
As $\nabla F_\alpha(U)=\nabla_vQ(A(U))\nabla A(U)$, thus we observe that
\begin{equation}
F_\alpha(u|U)=F_\alpha(u)-F_\alpha(U)-\nabla F_\alpha(U)[\nabla A(U)]^{-1}(A(u)-A(U))=Q_\alpha(A(u)|A(U)).
\end{equation}

\noindent
{\it Step 1.} 
Consider first the case $u\in D_\varepsilon, U\in D$. Observe that there exists a constant $c_1>0$ such that
\begin{align}\label{eta-squares}
\eta(u|U)&=H(A(u)|A(U)) \ge c_1|A(u)-A(U)|^2
\end{align}
where $c_1=\inf_{y\in D_\varepsilon}\nabla^2_vH(A(y))$, which is positive by the uniform convexity of $H$ on the set ${\rm Im}(A(D_\varepsilon))$.
Next we  estimate the relative flux as follows
\begin{align}\nonumber
|F_\alpha(u|U)|&=|Q_\alpha(A(u)|A(U))|\le \sup\limits_{y\in D_\varepsilon}\nabla^2_v(A(y))\ |A(u)-A(U)|^2\le c_2\eta(u|U).
\end{align}
where the constant $c_4$ includes $\sup_{y\in D_\epsilon}|\nabla^2F(y)|$ and  $\sup_{y\in D_\epsilon}|\nabla^2A(y)|$.

{\it Step 2.}
Let now $u\in  X\setminus D_\varepsilon$ and $U\in D$.
Observe that since $U$ is bounded then there exist constants $k_1, k_2$ such that 
\begin{equation}
\eta(u|U)=\eta(u)-\eta(U) -G(U)(A(u)-A(U))\ge \eta(u)-k_1-k_2|A(u)|. 
\end{equation}
Observe that for any $R>0$ 
\begin{equation}
|A(u)|\le \sup\limits_{|y|\le R}|A(y)|+\sup\limits_{|y|> R}\left\{\frac{|A(y)|}{\eta(y)}\right\}\eta(u)
\end{equation}
and thus by the first condition of \eqref{H4thanos} and as $A$ is continuous on $\overline X$ we can claim there exists a constant $k_3$ such that 
\begin{equation}\label{lemma-statement}
|A(u)|\le k_3+\frac{1}{2k_2}\eta(u) 
\end{equation}
and hence
\begin{equation}\label{eta-etarelative}
\eta(u|U)\ge \frac{1}{2}\eta(u)-c_5.
\end{equation}
We further estimate using \eqref{H4thanos}$_2$ 
\begin{align}\nonumber
|F_\alpha(u|U)|&=|F_\alpha(u)-F_\alpha(U)-\nabla F_\alpha(U)\nabla A(U)^{-1}(A(u)-A(U))|\\\label{f-rel}
&\le|F_\alpha(u)|+K_1|A(u)|+K_2\le c_6(1+\eta(u)).
\end{align}
Observe then that from \eqref{f-rel} together with \eqref{eta-etarelative}, we conclude
\begin{equation}\label{a12}
|F_\alpha(u|U)|\le c(1+\eta(u|U)).
\end{equation}
We shall denote by  $u^*\in \partial D_\varepsilon$  such a vector that there exists $t^*\in (0,1)$ such that
$u^*=(1-t^*)u+t^*U$. Obviously $|u^*-U|\ge \varepsilon$. Thus since $\nabla A$ is nonsingular then there exists $\tilde\varepsilon>0$ such that $|A(u^*)-A(U)|\ge \tilde\varepsilon$.
Notice that the function 
$$\R_+\ni t\mapsto H((t(A(u)-A(U))+A(U))|A(U))$$ is monotone and hence $H(A(u)|A(U))\ge H(A(u^*)|A(U))$.
As we have already shown in \eqref{eta-squares} we have that $H(A(u^*)|A(U))\ge c_1 |A(u^*)-A(U)|^2$, consequently we obtain that 
$$\eta(u|U)=H(A(u)|A(U))\ge c_1\tilde\varepsilon^2.$$
Thus we conclude \eqref{rel-est-appendix} for all $u\in X$ from \eqref{a12}. By the continuity of 
$F_\alpha(\cdot|U)$ and $\eta(\cdot|U)$ condition $\eqref{rel-est-appendix}$ holds for all $u\in \overline X$.
%
\end{proof}
For reader's convenience we recall here the slicing lemma (cf.~\cite[Theorem 1.5.1]{Evans-weak}), which is used for showing desintegration of the concentration measure. Let then $\mu$ be a finite, nonnegative Radon measure on $\R^{n+m}$ and let $\sigma$ be the canonical projection of $\mu$ onto $\R^n$, which means that $\sigma(E)\equiv
\mu(E\times\R^m)$ for each Borel set $E\subset\R^n$.  
\begin{lemma}\label{slicing}
For $\sigma-a.e$ point $x\in\R^n$ there exists a Radon probablity measure $\nu_x$ on $\R^m$, such that 
\begin{itemize}
\item[(i)] the mapping $ x\mapsto\int_{\R^n}f(x,y)d\nu_x(y)$ is  $\sigma-$measurable \\
and
\item[(ii)] $\int_{\R^{n+m}}f(x,y)d\mu(x,y)=\int_{\R^n}\left(\int_{\R^m}f(x,y)d\nu_x(y)\right)d\sigma(x)$ for each bounded, continuous $f$. 
\end{itemize}
\end{lemma}
\begin{remark}\label{rem-slicing}
Note that in case we consider measures associated to sequences which are bounded in one of the variables, we can even claim that the corresponding canonical projection is absolutely continuous with respect to the Lebesgue measure. In the case considered in the current paper we  deal with a domain $[0,T]\times\T^d$. Considered sequences are bounded in $L^\infty(0,T; L^1(\T^d))$. Then the corresponding concentration measure $m$ admits a desintegration of the form
\begin{equation}
m=m^t(dx)\otimes dt,
\end{equation}
where $t\mapsto m^t$ is bounded and weak-star measurable as a map from $[0,T]$ to ${\mathcal M}^+(\T^d)$. 
\end{remark}

\section{Convex functions} \label{s:convex}
We include here the facts used to show that compressible Euler system satisfies the assumptions of the main theorem. We consider a convex function $M:\R_+\to\R_+$, which is continuous, $M(v)=0$ iff $v=0$ and
\begin{equation}
\lim\limits_{v\to0}\frac{M(v)}{v}=0, \quad \lim\limits_{v\to\infty}\frac{M(v)}{v}=\infty.
\end{equation}
A function satisfying the above properties is called an $N-$function. 
We define  the Fenchel conjugate to $M$ as $M^*(\xi):=\sup_{\rho}(\xi\cdot\rho-M(\rho))$. For functions $v,w:\Omega\to\R_+$ such that 
$\int_{\Omega}M(v)dx<\infty$ and $\int_{\Omega}M^*(w)dx<\infty$ the following estimate (called Fenchel-Young inequality) holds
\begin{equation}
vw\le M(v)+M^*(w).
\end{equation}
To compare $N-$functions we will say that $M_1$ is essentially stronger than $M_2$ if $M_2(v)\le M_1(av)$  for all $v\ge v_0\ge0$ for all $a>0$ and some $v_0(a)$. 
For the purpose of estimates in Section~\ref{App} we will use the following lemma.
\begin{lemma}
Let $M_i, M^*_i$, $i=1,2$ be two complementary pairs of $N-$functions. Then the following conditions are equivalent. 
\begin{enumerate}
\item[$(i)$] $M_1$ is essentially stronger than $M_2$;
\item[$(ii)$] $M_2^*$ is essentially stronger than $M_1^*$;
\item[$(iii)$] $\forall\lambda>0$, $\lim\limits_{v\to\infty}\frac{M_2(\lambda v)}{M_1(v)}=0. $
\end{enumerate} 
\end{lemma}
The proof of the above fact follows from simple estimates, see~\cite[Chapter 2.2.,~Th.~2]{Rao} for details. 

\bibliographystyle{abbrv}
\bibliography{hyper} 

\begin{thebibliography}{10}

\bibitem{AlBo}
J.~J. Alibert and G.~Bouchitt{\'e}.
\newblock Non-uniform integrability and generalized {Y}oung measures.
\newblock {\em J. Convex Anal.}, 4(1):129--147, 1997.

\bibitem{ball}
J.~M. Ball.
\newblock A version of the fundamental theorem for {Y}oung measures.
\newblock In {\em P{DE}s and continuum models of phase transitions ({N}ice,
  1988)}, volume 344 of {\em Lecture Notes in Phys.}, pages 207--215. Springer,
  Berlin, 1989.

\bibitem{BeFeNo}
P.~Bella, E.~Feireisl, and A.~Novotn{\'y}.
\newblock Dimension reduction for compressible viscous fluids.
\newblock {\em Acta Appl. Math.}, 134:111--121, 2014.

\bibitem{BrDeLeSz2011}
Y.~Brenier, C.~De~Lellis, and L.~Sz{{\'e}}kelyhidi, Jr.
\newblock Weak-strong uniqueness for measure-valued solutions.
\newblock {\em Comm. Math. Phys.}, 305(2):351--361, 2011.

\bibitem{BrFe}
J.~B\v{r}ezina and E.~Feireisl.
\newblock Measure-valued solutions to the complete {E}uler system.
\newblock {\em https://arxiv.org/abs/1702.04870}, 02 2017.

\bibitem{BrKrMa}
J.~B\v{r}ezina, O.~Kreml, and V.~M\'{a}cha.
\newblock Dimension reduction for the full {N}avier-{S}tokes-{F}ourier system.
\newblock {\em J. Math. Fluid Mech.}, 19(4):659--683, 2017.

\bibitem{CleoTz}
C.~Christoforou and A.~Tzavaras.
\newblock Relative entropy for hyperbolic-parabolic systems and application to
  the constitutive theory of thermoviscoelasticity.
\newblock {\em to appear in Arch. Ration. Mech. Anal.},
  https://arxiv.org/abs/1603.08176, 03 2016.

\bibitem{Da2000}
C.~M. Dafermos.
\newblock {\em Hyperbolic conservation laws in continuum physics}, volume 325
  of {\em Grundlehren der Mathematischen Wissenschaften [Fundamental Principles
  of Mathematical Sciences]}.
\newblock Springer-Verlag, Berlin, 2000.

\bibitem{Dafermos1985}
C.~M. Dafermos and W.~J. Hrusa.
\newblock Energy methods for quasilinear hyperbolic initial-boundary value
  problems. {A}pplications to elastodynamics.
\newblock {\em Arch. Rational Mech. Anal.}, 87(3):267--292, 1985.

\bibitem{tmna}
T.~Debiec, P.~Gwiazda, K.~{\L}yczek, and A.~{\'S}wierczewska-Gwiazda.
\newblock Relative entropy method for measure-valued solutions in natural
  sciences.
\newblock {\em to appear in Topol. Meth. Nonlinear Analysis}, 09 2017.

\bibitem{DeStTz2}
S.~Demoulini, D.~M.~A. Stuart, and A.~E. Tzavaras.
\newblock A variational approximation scheme for three-dimensional
  elastodynamics with polyconvex energy.
\newblock {\em Arch. Ration. Mech. Anal.}, 157(4):325--344, 2001.

\bibitem{tzavaras1}
S.~Demoulini, D.~M.~A. Stuart, and A.~E. Tzavaras.
\newblock Weak-strong uniqueness of dissipative measure-valued solutions for
  polyconvex elastodynamics.
\newblock {\em Arch. Ration. Mech. Anal.}, 205(3):927--961, 2012.

\bibitem{DiPerna}
R.~J. DiPerna.
\newblock Measure-valued solutions to conservation laws.
\newblock {\em Arch. Rational Mech. Anal.}, 88(3):223--270, 1985.

\bibitem{DiMa87}
R.~J. DiPerna and A.~J. Majda.
\newblock Oscillations and concentrations in weak solutions of the
  incompressible fluid equations.
\newblock {\em Comm. Math. Phys.}, 108(4):667--689, 1987.

\bibitem{Evans-weak}
L.~C. Evans.
\newblock {\em Weak convergence methods for nonlinear partial differential
  equations}, volume~74 of {\em CBMS Regional Conference Series in
  Mathematics}.
\newblock Published for the Conference Board of the Mathematical Sciences,
  Washington, DC; by the American Mathematical Society, Providence, RI, 1990.

\bibitem{Evans}
L.~C. Evans and R.~F. Gariepy.
\newblock {\em Measure theory and fine properties of functions}.
\newblock Textbooks in Mathematics. CRC Press, Boca Raton, FL, revised edition,
  2015.

\bibitem{FGSW16}
E.~Feireisl, P.~Gwiazda, A.~{{\'S}wierczewska-Gwiazda}, and E.~Wiedemann.
\newblock Dissipative measure-valued solutions to the compressible
  {N}avier-{S}tokes system.
\newblock {\em Calc. Var. Partial Differential Equations}, 55(6):Art. 141, 20,
  2016.

\bibitem{FeJiNo}
E.~Feireisl, B.~J. Jin, and A.~Novotn{{\'y}}.
\newblock Relative entropies, suitable weak solutions, and weak-strong
  uniqueness for the compressible {N}avier-{S}tokes system.
\newblock {\em J. Math. Fluid Mech.}, 14(4):717--730, 2012.

\bibitem{FMT16}
U.~S. Fjordholm, S.~Mishra, and E.~Tadmor.
\newblock On the computation of measure-valued solutions.
\newblock {\em Acta Numer.}, 25:567--679, 2016.

\bibitem{GiTz}
J.~Giesselmann and A.~E. Tzavaras.
\newblock Singular limiting induced from continuum solutions and the problem of
  dynamic cavitation.
\newblock {\em Arch. Ration. Mech. Anal.}, 212(1):241--281, 2014.

\bibitem{GSW2015}
P.~Gwiazda, A.~{\'S}wierczewska-Gwiazda, and E.~Wiedemann.
\newblock Weak-strong uniqueness for measure-valued solutions of some
  compressible fluid models.
\newblock {\em Nonlinearity}, 28(11):3873--3890, 2015.

\bibitem{GwiWie}
P.~Gwiazda and E.~Wiedemann.
\newblock Generalized entropy method for the renewal equation with measure
  data.
\newblock {\em Commun. Math. Sci.}, 15(2):577--586, 2017.

\bibitem{landau}
L.~D. Landau and E.~M. Lifshitz.
\newblock {\em Course of theoretical physics. {V}ol. 6}.
\newblock Pergamon Press, Oxford, second edition, 1987.
\newblock Fluid mechanics, Translated from the third Russian edition by J. B.
  Sykes and W. H. Reid.

\bibitem{majda}
A.~Majda.
\newblock {\em Compressible fluid flow and systems of conservation laws in
  several space variables}, volume~53 of {\em Applied Mathematical Sciences}.
\newblock Springer-Verlag, New York, 1984.

\bibitem{MiMiPe2}
P.~Michel, S.~Mischler, and B.~t. Perthame.
\newblock General entropy equations for structured population models and
  scattering.
\newblock {\em C. R. Math. Acad. Sci. Paris}, 338(9):697--702, 2004.

\bibitem{MiMiPe}
P.~Michel, S.~Mischler, and B.~t. Perthame.
\newblock General relative entropy inequality: an illustration on growth
  models.
\newblock {\em J. Math. Pures Appl. (9)}, 84(9):1235--1260, 2005.

\bibitem{Pe2007}
B.~Perthame.
\newblock {\em Transport equations in biology}.
\newblock Frontiers in Mathematics. Birkh{\"a}user Verlag, Basel, 2007.

\bibitem{Pro}
G.~Prodi.
\newblock Un teorema di unicit\`a per le equazioni di {N}avier-{S}tokes.
\newblock {\em Ann. Mat. Pura Appl. (4)}, 48:173--182, 1959.

\bibitem{Rao}
M.~M. Rao and Z.~D. Ren.
\newblock {\em Theory of {O}rlicz spaces}, volume 146 of {\em Monographs and
  Textbooks in Pure and Applied Mathematics}.
\newblock Marcel Dekker, Inc., New York, 1991.

\bibitem{Ser}
J.~Serrin.
\newblock The initial value problem for the {N}avier-{S}tokes equations.
\newblock In {\em Nonlinear {P}roblems ({P}roc. {S}ympos., {M}adison, {W}is.
  1962)}, pages 69--98. Univ. of Wisconsin Press, Madison, Wis., 1963.

\bibitem{T79}
L.~Tartar.
\newblock Compensated compactness and applications to partial differential
  equations.
\newblock In {\em Nonlinear analysis and mechanics: {H}eriot-{W}att
  {S}ymposium, {V}ol. {IV}}. Pitman, Boston, Mass.-London, 1979.

\bibitem{Wiedemann}
E.~Wiedemann.
\newblock Weak-strong uniqueness in fluid dynamics.
\newblock {\em https://arxiv.org/abs/1705.04220}, 05 2017.

\bibitem{young}
L.~Young.
\newblock Generalized curves and the existence of an attained absolute minimum
  in the calculus of variations.
\newblock {\em Comptes Rendus des S{\'e}ances de la Soci{\'e}t{\'e} des
  Sciences et des Lettres de Varsovie}, 30:211--234, 1937.

\end{thebibliography}
\end{document}